\documentclass[10pt]{article}
\usepackage{graphicx}
\usepackage[a4paper, textwidth=16cm]{geometry}

\usepackage{amsthm,amsmath}
\usepackage{xcolor,paralist,hyperref,titlesec,fancyhdr,etoolbox}
\newtheorem{theorem}{Theorem}[]
\newtheorem{definition}[theorem]{Definition}
\newtheorem{lemma}[theorem]{Lemma}

\usepackage{amssymb}
\newtheorem{assumption}{Assumption}
\newtheorem{remark}{Remark}

\titleformat{\section}
  {\normalfont\Large\bfseries}
  {\thesection}
  {1em}
  {}
  
\titlespacing*{\section}{0pt}{0ex}{0ex}

\hypersetup{ colorlinks=true, linkcolor=black, filecolor=black, urlcolor=black, citecolor=black }

\usepackage{lipsum}

\begin{document}
\title{\textbf{The local Calder\'on problem and the determination at the boundary of a complex anisotropic admittivity}} 
\author{Jessica Crosse\thanks{Department of Mathematics and Statistics, University of Limerick, V94 T9PX, Limerick, Ireland. E-mail: jessica.crosse@ul.ie}
\and Romina Gaburro\thanks{Department of Mathematics and Statistics, University of Limerick, Health Research Institute (HRI), V94 T9PX, Limerick, Ireland. E-mail: romina.gaburro@ul.ie}}
\date{}
\maketitle

\begin{abstract}
We address Calder\'on's problem of stably determining the anisotropic complex admittivity $\sigma$ in a domain $\Omega\subset\mathbb{R}^n$, with $n\geq3$, representing a conducting medium, in terms of a Dirichlet-to-Neumann map locally prescribed on a non-empty portion $\Sigma$ of the boundary of $\Omega$, $\partial\Omega$. $\sigma$ is assumed to be of type $\sigma(\cdot)=A(\cdot,a(\cdot))$ in $\Omega$, where the one-parameter family of complex-symmetric matrices $[\lambda^{-1},\:\lambda]\ni t\mapsto A(\cdot,\: t)$ is assumed to be \textit{a-priori} known and the scalar function $a$ is unknown. We establish Lipschitz and H\"older stability estimates at the boundary for $\sigma$ and its derivatives of arbitrary order on $\Sigma$, respectively, in terms of the local map. 
\end{abstract} 



\section{Introduction}
\label{sec:intro}
In this paper, we address the inverse problem of stably determining the complex anisotropic \textit{admittivity} $\sigma$ in a domain $\Omega\subset\mathbb{R}^n$, with $n\geq 3$, from the Dirichlet-to-Neumann map localised on a portion $\Sigma$ of the boundary of $\Omega$, $\partial\Omega$. If $\Omega$ is occupied by a medium having admittivity $\sigma$, in absence of internal sources, the electrostatic potential $u$ in $\Omega$ satisfies 
\begin{equation}
    \label{eqn: 1}
    \textrm{div}(\sigma\nabla u)=0 ,\quad\textrm{in}\quad\Omega,
\end{equation}
where the (possibly anisotropic) electric admittivity at frequency $k$ is given by the complex-symmetric matrix valued function 
\begin{equation}\label{complex sigma}
\sigma(x)=\sigma_R(x) + ik\sigma_I(x),\qquad x\in\Omega, 
\end{equation}
where $\sigma_R$ is the \textit{conductivity} of $\Omega$ and $\sigma_I$ is the \textit {permittivity} of $\Omega$. 

This problem, also known as Calder\'on's problem, or Electrical Impedance Tomography (EIT), arises in many different fields such as geophysics, medicine, and nondestructive testing of materials. When $k=0$ and $\sigma$ is real, the inverse problem is known as the inverse conductivity problem and its first mathematical formulation is due to the seminal paper of Calder\'{o}n \cite{C}, where he addressed the problem of whether the isotropic conductivity $\sigma = \gamma I$ can be uniquely determined by the Dirichlet-to-Neumann (D-N) map
$$
\Lambda_{\sigma}:\,{H}^{\frac{1}{2}}(\partial\Omega)\,\ni\,u\vert_{\partial\Omega}\,
\rightarrow\,{\sigma}\nabla{u}\cdot\nu\vert_{\partial\Omega}\in{H}^{-\frac{1}{2}}(\partial\Omega), 
$$
where $u$ solves \eqref{eqn: 1} and $\nu$ denotes the outward unit normal to the boundary $\partial\Omega$. We refer to \cite{A1, Koh-V1, Koh-V2, N, Sy-U} as seminal contributions to this inverse problem in the real case and to \cite{Bo, U} for an overview on it.



The complex case has been less studied to date and in this respect we recall, in the isotropic case, the two-dimensional uniqueness results of \cite{Bu, Fr}, and those of Lipschitz stability in \cite{BerFra11}. For the complex anisotropic case, we recall the recent global stability result \cite{FosGabSin25} for the case when $\sigma = a A$, with $A$ an \textit{a-priori} known real Lipschitz matrix-valued function and $a$ an unknown affine complex scalar function to be determined. Here we consider the case where $\sigma$ is a complex-symmetric matrix-valued function of type $\sigma = A(\cdot,a(\cdot))$, where $t\mapsto A(\cdot, t)$ is \textit{a-priori} known and the dependence of $A$ (or $\sigma$) on the scalar unknown function $a$ can be non-linear. 

We emphasize that allowing $\sigma$ in \eqref{eqn: 1} to be a complex-valued matrix, accounting for both the conductivity $\sigma_R$ and the permittivity $\sigma_I$ in $\Omega$ at a fixed frequency $k$, yields a more realistic model of the material’s electrical behavior than the purely real case $\sigma = \sigma_R$. In fact, both $\sigma_R$ and $\sigma_I$ are essential for accurately characterising the electrical material properties via externally applied electric fields. When current is injected into $\Omega$ through electrodes on $\partial\Omega$, an electrical field is generated and induces ionic motion, corresponding to conductivity. At the same time, under a fixed frequency $k$, ionic polarisation occurs, which is captured by the permittivity. For more on the appropriateness of the model \eqref{eqn: 1} having $\sigma$ complex, we refer to \cite{AdGabLio15} and \cite{FosGabSin25}.

The admittivity $\sigma$ considered in this paper is anisotropic. This is motivated by the fact that anisotropy is commonly present in nature, e.g., in the human body; in the theory of homogenization; as a result of deformation of an isotropic material. Moreover, many biological tissues, such as the breast, muscles, and brain, include microscopic fiber structures that provide anisotropic physical characteristics on a larger scale, see \cite{AdGabLio15}. 

Even in the real case, Tartar's observation (see \cite{Koh-V1}) showed that, in the general case, an anisotropic $\sigma$ is not uniquely determined by $\Lambda_{\sigma}$ (full measurements). In fact, if $\psi:\overline{\Omega}\longrightarrow\overline{\Omega}$ is a diffeomorphism that keeps the boundary $\partial\Omega$ fixed, i.e., $\psi\vert_{\partial\Omega} = I$, then $\sigma$ and its push-forward under $\psi$,
\[\psi^{*} = \frac{(D\psi)\sigma({D\psi})^T}{\det(D\psi)}\circ\psi^{-1}\]
produce the same D-N map. Since this observation, a line of research initiated by the seminal paper of Lee and Uhlmann \cite{Le-U}, has been that of investigating the determination of $\sigma$ modulo a change of variables that fixes the boundary (see \cite{La-U, La-U-T} and \cite{Astala2005,Sylvester1990, Sy-U} for the two-dimensional case). Another direction has been the one of assuming that the anisotropic conductivity $\sigma$ depends on a finite number of unknown spacially-dependent parameters that uniquely identify $\sigma$. In this direction, we refer to \cite{A, A-G1, A-G2, AGS, Koh-V1, Lio} and the line of research initiated in \cite{AleVes05}.

Here we address the issue of stability at the boundary $\partial\Omega$ for the inverse problem of determining the complex anisotropic admittivity $\sigma$ from a local D-N map, $\Lambda_{\sigma}^{\Sigma}$, localised on an open portion $\Sigma\subset\partial\Omega$. It is well established (see \cite{A}) that, even in the real isotropic case ($\sigma=\gamma I$, where $\gamma$ is a real-valued scalar function and $I$ represents the $n\times n$ identity matrix), the optimal stability of $\sigma$ in the interior of $\Omega$ with respect to $\Lambda_{\sigma}$ or its local version $\Lambda_{\sigma}^{\Sigma}$, is logarithmic (see also \cite{B-B-R, B-F-R, Liu}). On the other hand, Lipschitz stability in the interior can be restored if $\sigma$ is deemed to belong to a finite dimensional space (see \cite{AleVes05} and the subsequent developments in \cite{AleDeHGabSin17, BerFra11, FosGabSin21, FosGabSin25, G-S} and \cite{Garde2024}). 

Following this line of research, we consider a complex anisotropic admittivity of type $\sigma(x)=A\big(x,a(x)\big)$, $x\in\Omega$, where the one-parameter family of complex-symmetric matrices $[\lambda^{-1}, \lambda]\ni t\rightarrow A(\cdot, t)$ is assumed to be \textit{a-priori} known and satisfying the monotonicity assumption
\begin{equation}\label{monotonicity complex}
\Re\big(D_tA(\cdot,t)\big)\geq CI>0,
\end{equation}
for a positive constant $C$. $a(\cdot)$ is an unknown scalar function to be stably determined in terms of $\Lambda_{\sigma}^{\Sigma}$. The precise assumptions and formulation of the problem, are given in Section \ref{sec2}.

Within this setting, we stably determine $a$ and its derivatives $D^h a$ (hence $A$ and $D^hA$) of any arbitrary order $h$, $h\geq1$, on $\Sigma$, with respect to $\Lambda_{\sigma}^{\Sigma}$ (see Subsection \ref{sec2.2} for the precise formulation of the local map). Specifically, we show that, under suitable conditions, $A$ and $D^hA$, at the boundary $\partial\Omega$, depend upon $\Lambda_{\sigma}^\Sigma$ with a modulus of continuity of Lipschitz and H\"older type, respectively, provided that the frequency $k$ in \eqref{complex sigma} lies in an interval that explicitly depends on the \textit{a-priori} information (see Theorems \ref{thm 1}, \ref{der thm}). 

Our results rely on the existence of singular solutions to \eqref{eqn: 1} having an isolated singularity (of any arbitrary order) outside $\Omega$. The existence of such solutions was first proved in \cite{A} (see also \cite{I}) for the case when $\sigma$ is a real matrix-valued function belonging to $W^{1,p}(\Omega)$, with $p>n$. These solutions were employed in the same paper to establish a Lipschitz boundary stability estimate of a real anisotropic conductivity of type $\sigma=A(a)$ and a H\"older boundary stability estimate for its derivatives of any arbitrary order. These results were obtained within the setting of having $[\lambda^{-1}, \lambda]\ni t\rightarrow A(t)$ \textit{a-priori} known, with $a$ the unknown scalar function to be estimated, under a monotonicity condition on $A(\cdot)$ of the form
\[D_tA\geq CI>0,\]
for some positive constant C. These results were extended in \cite{A-G1, A-G2, G-L} to the more general real case of $\sigma=A(\cdot, a)$ in terms of global and local measurements and to the case when $\Omega$ is a manifold (with global measurements). 

Similar results were subsequently established in the context of the companion inverse problem of diffuse optical tomography (DOT), where \eqref{eqn: 1} is replaced by $-\textrm{div}(\sigma\nabla u)+qu=0$, for the case of $\sigma$ and $q$ real \cite{G}, and complex (see \cite{Cu-G-N, Cu-G-N-S, Do}). The common feature in \cite{Cu-G-N, Cu-G-N-S, Do} was the availability of an explicit formula displaying the dependence of $\sigma$ on $a$ (playing the important physiological role in DOT of the so-called absorption coefficient), making the inverse problem in DOT of establishing stability estimates of $\sigma$ and its derivatives at the boundary more tractable. We also refer to \cite{Br, KaYu, NaTa1, NaTa2, NaTa3}, where the issues of stability and reconstruction of real isotropic conductivities at the boundary was studied. For the real anisotropic case, we recall the more recent local boundary stability estimate of H\"older type in terms of $\Lambda_{\sigma}^{\Sigma}$, established in \cite{AGS}, in the case when $\sigma$ is assumed to be constant near $\Sigma\subset\partial\Omega$.

Here, we use the method of the singular solutions introduced in \cite{A} to prove a boundary Lipschitz stability estimate for a complex conductivity $\sigma$ of type $\sigma = A(\cdot, a)$ in  \eqref{eqn: 1} together with a boundary H\"older stability estimate of the derivatives of $\sigma$, of any arbitrary order, under suitable conditions on $A$ and $a$ that include that in \eqref{monotonicity complex}. As in \cite{Cu-G-N, Cu-G-N-S, Do}, the existence of singular solutions rely on the existence of the Green's function for the forward operator in \eqref{eqn: 1}. Since $\Re\sigma\geq CI>0$, the complex equation in \eqref{eqn: 1} can be treated as a two by two strongly elliptic system with real coefficients, which Green matrix's existence is guaranteed by the $W^{1,p}$- regularity of $\sigma$, with $p>n$ (see \cite[Section 3.1]{FosGabSin25}). Hence we show existence of singular solutions to \eqref{eqn: 1} having the same type of isolated singularities of those in \cite{Cu-G-N, Cu-G-N-S} but also having trace at the boundary compactly supported in $\Sigma$. This suits the fact that the D-N map $\Lambda_{\sigma}^{\Sigma}$ is localised to $\Sigma\subset\partial\Omega$ and allows, in turn, to establish our stability estimates locally on $\Sigma$, in terms of $\Lambda_{\sigma}^{\Sigma}$.

We improve upon the results obtained in \cite{Cu-G-N, Cu-G-N-S, Do} under the following aspects:
\begin{enumerate}
\item The \textit{a-priori} anisotropic structure of $\sigma$ considered here is more general.
\item The stability estimates are established in terms of local data (Theorems \ref{thm 1}, \ref{der thm}).
\item In Theorem \ref{der thm} an explicit interval of variability for the frequency $k$, in terms of the \textit{a-priori} data, is provided to allow for the boundary stable determination of the derivatives of $\sigma$  in terms of $\Lambda_{\sigma}^{\Sigma}$. This is based on an estimate from below for the gradient of the singular solutions of Theorem \ref{sing thm}, near the isolated singularity (Lemma \ref{Lem 4.1}), which is independent from the frequency $k$. In \cite{Do} an explicit interval for $k$ was provided to stably determine the boundary values of $\sigma$ in terms of the global measurements. In \cite{Cu-G-N, Cu-G-N-S}, the boundary stable determination of the derivatives of $\sigma$ was established under a low frequency regime, without quantifying the upper bound for $k$. 
\end{enumerate}

Our results represent the first step in an effort to prove stability for the complex anisotropic conductivity $\sigma=A(\cdot,a(\cdot))$, where $a$ could be piecewise affine in $\Omega$. Since the work of Kohn and Vogelius \cite{KV2, Koh-V1} and Alessandrini \cite{A1, A}, it is customary to treat the uniqueness and stability at the boundary, as a first step towards determination in the interior. It is in fact reasonable to expect that, likewise the stable determination of a piecewise affine $a$ on a given partition of $\Omega$ in the complex anisotropic conductivity $\sigma=aA$ treated in \cite{FosGabSin21}, the proof of a global stability estimate for the case $\sigma=A(\cdot,a)$ treated here, would be based on an iterative determination of boundary values and normal derivatives of $a$ (hence of $\sigma$) at the various interfaces of the domain partition, where $a$ is assumed to be piecewise affine. 

The paper is organised as follows. In Section \ref{sec2} we rigorously formulate the problem and state the main results (Theorems \ref{thm 1}, \ref{der thm}) of Lipschitz stability of $\sigma$ and H\"older stability of $D^h\sigma$ on $\Sigma$ in terms of $\Lambda_{\sigma}^{\Sigma}$, which is rigorously defined in Subsection \ref{sec2.2}. Section \ref{sec3} is devoted to proving existence of singular solutions of \eqref{eqn: 1} having an isolated singularity of any arbitrary order and trace compactly supported in $\Sigma$ (Theorem \ref{sing thm}). The proofs of Theorems \ref{thm 1}, \ref{der thm} are presented in Section \ref{sec4}. This section also contains the estimate from below of the gradient of the singular solutions constructed in Theorem \ref{sing thm} which is independent from $k$ (Lemma \ref{Lem 4.1}), a crucial step to the proof of Theorem \ref{der thm}.


\section{Formulation of the problem}\label{sec2}


\subsection{Main assumptions}\label{sec2.1}
We rigorously formulate the problem by introducing the following notation, definitions and assumptions. For $n\geq3$, a point $x\in\mathbb{R}^n$ will be denoted by $x=(x',x_n)$, where $x'\in\mathbb{R}^{n-1}$ and $x^n\in\mathbb{R}$. Moreover, given a point $x\in\mathbb{R}^n$, we will denote with $B_r(x)$, $B'_r(x')$ the open balls in $\mathbb{R}^n$, $\mathbb{R}^{n-1}$, centred at $x$ and $x'$, respectively, with radius $r$ and by $Q_r(x)$ the cylinder
\[Q_r(x)=B'_r(x')\times (x_n-r,x_n+r).\]
For any $v,w\in\mathbb{C}^n$, with $v=(v_1,...,v_n),\;w=(w_1,...,w_n)$, we will denote by $v\cdot w$, the expression
\[v\cdot w=\sum_{i=1}^n v_iw_i.\]
\begin{definition} \label{def: lips bound}
    We shall say that the boundary of $\Omega$, denoted by $\partial\Omega$, is of Lipschitz class with constants $r_0, L>0$, if for any $P\in\partial\Omega$, there exists a rigid transformation of coordinates under which we have $P=0$ and 
    \[\Omega\cap Q_{r_0}(0)=\big\{(x',x_n)\in Q_{r_0}(0)\;\big|\;x_n>\varphi(x')\big\},\]
    where $\varphi$ is a Lipschitz function on $B'_{r0}(0)$ satisfying
    \[\varphi(0)=0\] and \[\Vert\varphi\Vert_{C^{0,1}(B'_{r_0}(0))}\leq Lr_0.\]
\end{definition} 
Throughout the entire manuscript, $\Omega$ is a bounded domain in $\mathbb{R}^n$, $n\geq 3$, with Lipschitz boundary with constants $L$, $r_0$ as per Definition \ref{def: lips bound}. 
\\We consider, for a fixed frequency $k>0$, the one-parameter family of complex matrix-valued functions
\begin{align}
    \label{eqn: A def}
    A\big(x,t\big)=A_R\big(x,t\big)+ik A_I\big(x,t\big),\quad&\textrm{for any}\quad x\in\Omega,\quad\textrm{for any}\quad t\in[\lambda^{-1},\lambda].
\end{align}
Here $A_R$, $A_I\in L^{\infty}(\Omega\times[\lambda^{-1},\lambda], Sym_n)$ are real matrix-valued functions where $Sym_n$ denotes the class of $n\times n$ real-valued symmetric matrices. We assume that $A_R$, $A_I$ commute. Moreover, throughout the entire manuscript we fix a real number $p>n$.
\begin{definition}\label{A in H}
    Given positive constants $\lambda,\mathcal{E}_1,\mathcal{E}_2,E,\mathcal{D}$, we say that $A(\cdot,\cdot)\in\mathcal{H}$ if the following conditions hold:
    \begin{align}
        \label{eqn: A in W}
        A_R\in W^{1,p}(\Omega\times[\lambda^{-1},\lambda],Sym_n),\nonumber\\
        A_I\in W^{1,p}(\Omega\times[\lambda^{-1},\lambda],Sym_n),
    \end{align}
    \begin{align}
        \label{eqn: DtA in W}
        D_tA_R\in W^{1,p}(\Omega\times[\lambda^{-1},\lambda],Sym_n),\nonumber\\
        D_tA_I\in W^{1,p}(\Omega\times[\lambda^{-1},\lambda],Sym_n),
    \end{align}
    \begin{align}
        \label{eqn: bound on norms of A}
        supess_{t\in[\lambda^{-1},\lambda]} \Big(&\Vert A(\cdot,t)\Vert_{L^p(\Omega)}+\Vert D_xA(\cdot,t)\Vert_{L^p(\Omega)}\nonumber\\
        &+\Vert D_tA(\cdot,t)\Vert_{L^p(\Omega)} +\Vert D_tD_xA(\cdot,t)\Vert_{L^p(\Omega)}\Big)\leq E,
    \end{align}
\begin{align}
 \label{eqn: mono}
     \Re\big(D_tA(x,t)\xi\cdot\xi\big)\geq \mathcal{D}^{-1}|\xi|^2,\quad&\textrm{for a.e.} \quad x\in\Omega,\nonumber\\
     \quad&\textrm{for every}\quad t\in[\lambda^{-1},\lambda],\quad\xi\in\mathbb{C}^n,
 \end{align}
 where, for $z\in\mathbb{C}$, $\Re(z)$ and $\Im(z)$ denote the real and imaginary part of $z$, respectively.
 \\We also assume that
  \begin{align}
     \label{eqn: AR assump}
     \mathcal{E}_1^{-1}|\xi|^2\leq A_R(x,t)\xi\cdot\xi\leq\mathcal{E}_1|\xi|^2,\quad&\textrm{for a.e.}\quad x\in\Omega,\nonumber\\
     \quad&\textrm{for every}\quad t\in[\lambda^{-1},\lambda],\quad\xi\in\mathbb{R}^n,
 \end{align}
  and $A_I$ satisfies either
 \begin{align}
     \label{eqn: AI1 assump}
     \mathcal{E}_2^{-1}|\xi|^2\leq A_I(x,t)\xi\cdot\xi\leq\mathcal{E}_2|\xi|^2,\quad&\textrm{for a.e.}\quad x\in\Omega,\nonumber\\
     \quad&\textrm{for every}\quad t\in[\lambda^{-1},\lambda],\quad\xi\in\mathbb{R}^n,
 \end{align}
 or
 \begin{align}
     \label{eqn: AI2 assump}
     -\mathcal{E}_2|\xi|^2\leq A_I(x,t)\xi\cdot\xi\leq-\mathcal{E}_2^{-1}|\xi|^2,\quad&\textrm{for a.e.}\quad x\in\Omega,\nonumber\\
     \quad&\textrm{for every}\quad t\in[\lambda^{-1},\lambda],\quad\xi\in\mathbb{R}^n.
 \end{align}
  We observe that \eqref{eqn: mono} is a condition of monotonicity for $A$ with respect to the variable $t$.
\end{definition}
We state below some facts which are straightforward consequences of Definition \ref{A in H}. As $A_R$ and $A_I$ commute, the real and imaginary parts of $A^{-1}(\cdot,\cdot)$, denoted by $A^{-1}_R(\cdot,\cdot)$ and $A_I^{-1}(\cdot,\cdot)$ respectively, are the symmetric, real matrix-valued functions on $\Omega\times[\lambda^{-1},\lambda]$ defined by
\begin{align}
A^{-1}_R(\cdot,\cdot)=&A_R(\cdot,\cdot)\big(A_R^2(\cdot,\cdot)+k^2A_I^2(\cdot,\cdot)\big)^{-1},\\
    A^{-1}_I(\cdot,\cdot)=&-k A_I(\cdot,\cdot)\big(A_R^2(\cdot,\cdot)+k^2A_I^2(\cdot,\cdot)\big)^{-1}.
\end{align}
As an immediate consequence of Definition \ref{A in H}, for every $\xi\in\mathbb{R}^n$,
\begin{equation}
 \label{eqn: Re A inv bound}
     \mathcal{E}_1^{-1}(\mathcal{E}_1^2+k^2\mathcal{E}_2^2)^{-1}|\xi|^2\leq A_R^{-1}(\cdot,\cdot)\xi\cdot\xi\leq\mathcal{E}_1(\mathcal{E}_1^{-2}+k^2\mathcal{E}_2^{-2})^{-1}|\xi|^2,\quad\textrm{on}\quad\Omega\times[\lambda^{-1},\lambda],
 \end{equation}
 and if $A_I(\cdot,\cdot)$ satisfies \eqref{eqn: AI1 assump}, then $A_I^{-1}$ is uniformly negative definite and satisfies
 \begin{equation}
 \label{eqn: AI inv bound1}
k\mathcal{E}_2(\mathcal{E}_1^{-2}\!+\!k^2\mathcal{E}_2^{-2})^{-1}|\xi|^2\leq A_I^{-1}(\cdot,\cdot)\xi\cdot\xi\leq -k\mathcal{E}_2^{-1}(\mathcal{E}_1^{2}\!+\!k^2\mathcal{E}_2^{2})^{-1}|\xi|^2,\quad\textrm{on}\quad\Omega\times[\lambda^{-1},\lambda],
 \end{equation}
 otherwise, if $A_I$ satisfies \eqref{eqn: AI2 assump}, then $A_I^{-1}$ is uniformly positive definite and satisfies
 \begin{equation}
    \label{eqn: AI inv bound2}
    k\mathcal{E}_2^{-1}(\mathcal{E}_1^2+k^2\mathcal{E}_2^2)^{-1}|\xi|^2\leq A_I^{-1}(\cdot,\cdot)\xi\cdot\xi\leq k\mathcal{E}_2(\mathcal{E}_1^{-2}+k^2\mathcal{E}_2^{-2})^{-1}|\xi|^2,\quad\textrm{on}\quad\Omega\times[\lambda^{-1},\lambda].
  \end{equation} 
Definition \ref{A in H} also implies that $A(\cdot,\cdot)$ satisfies the boundness condition
  \begin{align}
    \label{eqn: bound cond1}   |A_R|^2+|kA_I|^2\leq\mathcal{E}_1^2+k^2\mathcal{E}_2^2,\quad&\textrm{on}\quad\Omega\times[\lambda^{-1},\lambda].
  \end{align}
\begin{assumption}\label{assump on ai} 
We assume that $a\in L^\infty(\Omega)$ is a scalar function satisfying
    \begin{eqnarray}
      & &  \lambda^{-1}\leq a(x)\leq\lambda,\quad\textrm{for a.e.}\quad x\in\Omega,\label{eqn: a ass1}\\
  & &      \Vert a\Vert_{W^{1,p}(\Omega)}\leq \mathcal{F},\label{eqn: a ass2}
    \end{eqnarray}
    where $\mathcal{F}$ is a positive constant.
\end{assumption}
By denoting $u=u_1+iu_2$, the complex equation 
\begin{equation}\label{operator L}
    Lu=\textrm{div}\Big(A\big(\cdot,a(\cdot)\big)\nabla u(\cdot)\Big)=0,\quad \textrm{in}\quad\Omega,
\end{equation}
is equivalent to the system
\begin{equation}
    \label{eqn: system}
    \begin{cases}
        \textrm{div}\Big(A_R\big(\cdot,a(\cdot)\big)\nabla u_1(\cdot)\Big)-k \textrm{div}\Big(A_I\big(\cdot,a(\cdot)\big)\nabla u_2(\cdot)\Big)=0,\quad\textrm{in}\quad\Omega,\\
        k \textrm{div}\Big(A_I\big(\cdot,a(\cdot)\big)\nabla u_1(\cdot)\Big)+\textrm{div}\Big(A_R\big(\cdot,a(\cdot)\big)\nabla u_2(\cdot)\Big)=0,\quad\textrm{in}\quad\Omega,
    \end{cases}
\end{equation}
which can be written in compact form as
\begin{equation}
    \textrm{div}\Big(\mathcal{C}\big(\cdot,a(\cdot)\big)\nabla \mathbf{u}(\cdot)\Big)=0,\quad\textrm{in}\quad\Omega,
\end{equation}
where $\mathbf{u}=(u_1,u_2)^T$, and $\mathcal{C}$ is defined as
\begin{equation}
    \mathcal{C}\big(\cdot,a(\cdot)\big)=\begin{pmatrix}
A_R\big(\cdot,a(\cdot)\big) & -k A_I\big(\cdot,a(\cdot)\big)\\
k A_I\big(\cdot,a(\cdot)\big) & A_R\big(\cdot,a(\cdot)\big)
\end{pmatrix}.
\end{equation}
Observing that for $\xi=(\xi_1,\xi_2)^T\in\mathbb{R}^{2n}$, using the symmetry of $A_I$, we have
\begin{equation}\label{eqn: C dot}
\mathcal{C}\big(\cdot,a(\cdot)\big)\xi\cdot\xi=A_R\big(\cdot,a(\cdot)\big)\xi_1\cdot\xi_1+A_R\big(\cdot,a(\cdot)\big)\xi_2\cdot\xi_2.
\end{equation}
\eqref{eqn: AR assump}, together with \eqref{eqn: C dot} imply that the system \eqref{eqn: system} is uniformly elliptic and bounded, therefore it satisfies the strong ellipticity condition
\begin{equation}\label{eqn: C elliptic}
    C_1^{-1}|\xi|^2\leq \mathcal{C}\big(x,a(x)\big)\xi\cdot\xi\leq C_1|\xi|^2,\quad\textrm{for a.e}\quad x\in\Omega,\quad\textrm{for every}\quad \xi\in\mathbb{R}^{2n},
\end{equation}
where $C_1>0$ is a constant depending on $\mathcal{E}_1$.
\begin{definition}\label{def: local subsets}
Let $\Sigma$ be an open portion of $\partial\Omega$ and let $\partial\Sigma$ denote the boundary of $\Sigma$. For every $\eta$, $0<\eta<r_0$, where $r_0$ has been introduced in Definition \ref{def: lips bound}, we denote
    \begin{eqnarray}
        \Sigma_\eta &=&\big\{x\in\Sigma\;|\,\textrm{dist}(x, \partial\Sigma)>\eta\big\},\\
        U_\eta &=&\Big\{x\in\mathbb{R}^n\;|\;\textrm{dist}(x,\Sigma_\eta)<\frac{\eta}{4}\Big\},\\
        U_\eta^i &=&U_\eta\cap\Omega.
    \end{eqnarray}
\end{definition}
\begin{remark}\label{remark eta0}
    Given $\Sigma$ is open and non-empty, there exists $\eta_0$ such that for $0<\eta_0<r_0$, $\Sigma_{\eta_0}$ is always non-empty. From now on we shall only consider values of $\eta$ below $\eta_0$. We emphasize that $\eta_0$ is dependent on the choice of $\Sigma\subset\partial\Omega$. Therefore as $\Sigma$ becomes narrower, $\eta_0$ tends to 0 and the stability estimates worsen.
\end{remark}
\begin{definition}\label{apriori data}
    We will refer to the set of positive numbers $r_0, L, \lambda,\mathcal{E}_1, \mathcal{E}_2, E,\mathcal{D},$ $\mathcal{F}, \eta$ and $\eta_0$ introduced above, along with the space dimension $n,p>n$, the frequency $k$ and the diameter of $\Omega$, $diam(\Omega)$ as the \textit{a-priori} data.
\end{definition}


\subsection{The Dirichlet-to-Neumann map}\label{sec2.2}
Let $A\in\mathcal{H}$ and $a$ satisfies Assumption \ref{assump on ai} and let us define the following subspace of $H^\frac{1}{2}(\partial\Omega)$,
    \begin{equation}
        H^\frac{1}{2}_{co}(\Sigma)=\big\{f\in H^\frac{1}{2}(\partial\Omega)\;|\; \textrm{supp} f\subset\Sigma\big\},
    \end{equation}
    together with its closure in the $H^{\frac{1}{2}}(\partial\Omega)$-norm,
    \begin{equation}
        H_{00}^{\frac{1}{2}}(\Sigma)=\overline{H^\frac{1}{2}_{co}(\Sigma)}.
    \end{equation}
    Next, we rigorously define the Dirichlet-to-Neumann map localised to $\Sigma$. To emphasise its dependence on $a$ and $\Sigma$, we will denote it with $\Lambda^{\Sigma}_a$.
\begin{definition}
    The  local Dirichlet-to-Neumann map corresponding to $A\big(\cdot,a(\cdot)\big)$ and $\Sigma$ is the operator
\[\Lambda_a^\Sigma:H^{\frac{1}{2}}_{00}(\Sigma)\rightarrow (H^{\frac{1}{2}}_{00}(\Sigma))^\ast\subset H^{-\frac{1}{2}}(\partial\Omega),\]
defined by
\begin{equation}
    \label{eqn: DN map}
    \big\langle\Lambda_a^\Sigma f,\overline{g}\big\rangle=\int_\Omega A\big(x,a(x)\big)\nabla u(x)\cdot\nabla \psi(x)\;dx,
\end{equation}
for any $f,g\in H^{\frac{1}{2}}_{00}(\Sigma)$, where $u\in H^1(\Omega)$ is the weak solution to
\begin{equation*}
    \begin{cases}
        \textrm{div}\big(A\big(\cdot,a(\cdot)\big)\nabla u(\cdot)\big)=0, &\quad\textrm{in}\quad\Omega,\\
        u=f,&\quad\textrm{on}\quad\partial\Omega,
    \end{cases}
\end{equation*}
and $\psi\in H^1(\Omega)$ is any function such that $\psi|_{\partial\Omega}=g$ in the trace sense.
\end{definition}
We denote by $\langle\cdot,\cdot\rangle$ the $L^2(\partial\Omega)$-pairing between $H^{\frac{1}{2}}_{00}(\Sigma)$ and its dual $\big(H^{\frac{1}{2}}_{00}(\Sigma)\big)^\ast$.
\\Given $A\big(\cdot,\cdot\big)\in\mathcal{H}$ and $a_i$ satisfying Assumption \ref{assump on ai}, for $i=1,2$, the well-known Alessandrini's identity (see \cite{A})
\begin{equation}
    \label{eqn: Aless I}
    \Big\langle\big(\Lambda_{a_1}^\Sigma-\Lambda_{a_2}^\Sigma\big)f_1,\overline{f_2}\Big\rangle=\int_\Omega\Big(A\big(x,a_1(x)\big)-A\big(x,a_2(x)\big)\Big)\nabla u_1(x)\cdot\nabla u_2(x)\;dx
\end{equation}
holds true for any $f_i\in H^{\frac{1}{2}}_{00}(\Sigma)$, where $u_i\in H^1(\Omega)$ is the unique weak solution to the Dirichlet problem
\begin{equation*}
    \begin{cases}
        \textrm{div}\Big(A\big(\cdot,a_i(\cdot)\big)\nabla u_i(\cdot)\Big)=0,&\quad\textrm{in}\quad\Omega,\\
        u_i=f_i,& \quad\textrm{on}\quad\partial\Omega,
    \end{cases}
\end{equation*}
for $i=1,2$ respectively (see \cite{BerFra11} for example). We will denote 
\begin{equation}\label{eqn: DN norm def}
        \Vert\Lambda_{a}^\Sigma\Vert_\ast=\sup\Big\{|\langle\Lambda_{a}^\Sigma f,\overline{g}\rangle|\quad\Big|\; f,g\in H^{1/2}_{00}(\Sigma),\; \Vert f\Vert_{H^{1/2}_{00}(\Sigma)}=\Vert g\Vert_{H^{1/2}_{00}(\Sigma)}=1\Big\}
    \end{equation}
to be the norm on the Banach space of bounded linear operators between $H^{\frac{1}{2}}_{00}(\Sigma)$ and $\big(H^{\frac{1}{2}}_{00}(\Sigma)\big)^\ast$.


\subsection{Main result}\label{sec2.3}


\begin{theorem}[Local Lipschitz stability of boundary values]\label{thm 1}
   Let $\Omega$ be a bounded domain with Lipschitz constants $L,r_0$. Let $\Sigma$ be the open portion of $\partial\Omega$ with $\eta_0(\Sigma)>0$ introduced in Definition \ref{def: local subsets}. Assume that $k$ satisfies
    \begin{equation}
        \label{eqn: k1}
        0\!<k\!\leq\!\!\underset{\mathcal{A}+\mathcal{B}+\mathcal{C}=1}{\max}\!\Bigg\{\!\min \!\Bigg\{\frac{(m^3-m^{-3})\tan(\frac{\mathcal{A}\pi}{4})}{M^3-M^{-3}}\!, \; M^{-6}\tan\bigg(\frac{\mathcal{B}\pi}{2n}\bigg)\!,\; M^{-6}\tan\bigg(\frac{\mathcal{C}\pi}{2n}\bigg)\Bigg\}\!\Bigg\},
    \end{equation}
   where $M=\max\{\mathcal{E}_1,\mathcal{E}_2\}$, $m=\min\{\mathcal{E}_1,\mathcal{E}_2\}$, and $\mathcal{E}_1$, $\mathcal{E}_2$ have been introduced in Definition \ref{A in H}. If $A\in\mathcal{H}$ and $a_i$ is a real-valued function satisfying Assumption \ref{assump on ai}, for $i=1,2$, then we have
\begin{equation}
    \label{eqn: Stab result}
    \big\Vert A\big(x,a_1(x)\big)-A\big(x,a_2(x)\big)\big\Vert_{L^\infty(\overline{\Sigma}_\eta)}\leq C\big\Vert\Lambda_{a_1}^\Sigma-\Lambda_{a_2}^\Sigma\big\Vert_\ast,
\end{equation}
where $C$ is a positive constant which depends on the \textit{a-priori} data only.
\end{theorem}
\begin{remark}
    It should be noted that the range for $k$ such that \eqref{eqn: F assumptions} holds is valid for all positive values of $\mathcal{A},\mathcal{B}$, and $\mathcal{C}$ satisfying $\mathcal{A}+\mathcal{B}+\mathcal{C}\leq1$. The interval given in \eqref{eqn: k1} is chosen to optimise the range of feasible values of $k$. In particular, by choosing $\mathcal{A},\mathcal{B},\mathcal{C}=1/3$, \eqref{eqn: k1} simplifies to
    \begin{equation}
        0<k\leq \min\bigg\{\frac{(m^3-m^{-3})\tan\big(\frac{\pi}{12}\big)}{M^3-M^{-3}},\quad M^{-6}\tan\Big(\frac{\pi}{6n}\Big)\bigg\}.
    \end{equation}
\end{remark}

\begin{theorem}[Local H\"older Stability of derivatives at the boundary]\label{der thm}
       Let the hypotheses of Theorem \ref{thm 1} be satisfied. Furthermore, if there exists $E_h>0$ such that
   \begin{equation}
       A\in C^{h,\alpha}\big(\overline{U}_\eta\times[\lambda^{-1},\lambda]\big),
   \end{equation}
    \begin{equation}
        \label{eqn: A C norm}
        \Vert A\Vert_{C^{h,\alpha}(\overline{U}_\eta\times[\lambda^{-1},\lambda])}\leq E_h,
    \end{equation}
    \begin{equation}
        \Vert a_1-a_2\Vert_{C^{h,\alpha}(\overline{U}_\eta)}\leq E_h,
    \end{equation}
    for some $\alpha$, $0<\alpha<1$ and $\eta\leq\eta_0$, then we have
    \begin{equation}
        \Big\Vert D^h\Big(A\big(x,a_1(x)\big)-A\big(x,a_2(x)\big)\Big)\Big\Vert_{L^\infty(\overline{\Sigma}_\eta)}\leq C\big\Vert\Lambda_{a_1}^\Sigma-\Lambda_{a_2}^\Sigma\big\Vert_\ast^{\alpha\delta_h},
    \end{equation}
    where $h\geq1$ is an integer,
    \begin{equation}
        \delta_h=\underset{i=0}{\overset{h}{\prod}}\frac{\alpha}{\alpha+i}
    \end{equation}
    and $C$ is a positive constant which depends only on the \textit{a-priori} data and $h$.
\end{theorem}


\section{Singular solutions}\label{sec3}
We consider the operator introduced in \eqref{operator L},
\begin{equation}
    \label{eqn: EIT op}
    L=\textrm{div}\big(A\nabla\cdot\big)\quad\textrm{in}\quad B_R(z)=\big\{x\in\mathbb{R}^n\;|\;|x-z|<R\big\},
\end{equation}
with $A\big(\cdot,a(\cdot)\big)\in\mathcal{H}$ and $a$ satisfying Assumption \ref{assump on ai}.
 We recall the following result from \cite{Cu-G-N}.
 \begin{theorem}\label{Jasons thm}
     Given $L$ on $B_R(z)$ as in \eqref{eqn: EIT op}, for any $m=0,1,2,...$, there exists $u_m\in W^{1,p}(B_R(z)\backslash\{z\})$ such that 
     \begin{equation}\label{u solution}
     Lu_m=0\qquad\textnormal{in}\quad B_R(z)\backslash\{z\}, 
     \end{equation}
     with
     \begin{align}\label{u explicit}
         u_m(x)=&\big(A^{-1}\big(z,a(z)\big)(x-z)\cdot(x-z)\big)^{\frac{2-n-m}{2}}m!\Big(A^{-1}_{nn}\big(z,a(z)\big)\Big)^{\frac{m}{2}}\nonumber\\
        &\times C_m^{\frac{n-2}{2}}\Bigg(\frac{A^{-1}_{n}\big(z,a(z)\big)(x-z)}{\Big(A^{-1}_{nn}\big(z,a(z)\big)\Big)^{\frac{1}{2}}\big(A^{-1}\big(z,a(z)\big)(x-z)\cdot(x-z)\big)^\frac{1}{2}}\Bigg)+w(x),
     \end{align}
     where $C_m^{\frac{n-2}{2}}:\mathbb{C}\rightarrow\mathbb{C}$ is the complex Gegenbauer polynomial of degree $m$ and order $\frac{n-2}{2}$ and $A^{-1}_{n}\big(z,a(z)\big)$, $A^{-1}_{nn}\big(z,a(z)\big)$ denote the last row, last entry in the last row of the matrix $A^{-1}\big(z,a(z)\big)$, respectively. Moreover $w$ in \eqref{u explicit} satisfies
    \begin{equation}\label{eqn: w R1}
        |w(x)|+|x-z||Dw(x)|\leq C|x-z|^{2-n-m+\alpha},\quad \textrm{for any}\quad x\in B_R(z)\backslash\{z\}
    \end{equation}
    and
    \begin{equation}\label{eqn: w R2}
        \bigg(\int_{r<|x-z|<2r}|D^2w|^p\;dx\bigg)^{\frac{1}{p}}\leq Cr^{\frac{n}{p}-n+\alpha},\quad\textrm{for every}\quad r,\quad0<r<R/2.
    \end{equation}
    Here $\alpha$ is any number such that $0<\alpha<1-\frac{n}{p}$, and $C$ is a positive constant depending only on the a-priori data $n,p,\mathcal{E}_1,\mathcal{E}_2, k$, $E$ and on $\alpha$, $R$.
 \end{theorem}
 \begin{proof}
     See \cite[Theorem 2.2 and Claim 2.1]{Cu-G-N}.
 \end{proof}
 \begin{remark}
Assumption \eqref{eqn: AR assump} implies that for any $x,z\in\mathbb{R}^n$, with $x\neq z$, $\Re{\Big\{A^{-1}(z,a(z))(x-z)\cdot (x-z)\Big\}}>0$ and that without loss of generality we can assume that $A^{-1}_{nn}\big(z,a(z)\big)>0$.
 \end{remark}
 Next, we show the existence of singular solutions $u^{loc}_m$ having the same singular behaviour of $u_m$ constructed in Theorem \ref{Jasons thm}, equation \eqref{u explicit} and such that
\[u^{loc}_{m}\big|_{\partial\Omega}\in H^\frac{1}{2}_{00}(\Sigma).\]
From Definition \ref{def: local subsets} and Remark \ref{remark eta0}, for any $\eta$, $0<\eta\leq\eta_0$, we can construct a domain $\Omega_\eta$, with Lipschitz constants depending only on $\eta, r_0,L$ such that
\begin{equation}
    \Omega\subset\Omega_\eta, \quad \partial\Omega\cap\Omega_\eta\subset\subset\Sigma,
\end{equation}
and \begin{equation}
    \textrm{dist}(x,\partial\Omega_\eta)\geq\frac{\eta}{2},\quad\textrm{for every}\quad x\in U_\eta.
\end{equation}
\begin{theorem}[Singular solutions]\label{sing thm}
    Let $\Omega$ and $\Sigma$ be as in Theorem \ref{thm 1}. For any $\eta$, $0<\eta\leq\eta_0$, we fix an arbitrary point $z\in U_\eta$. For any $m=0,1,2,...$, there exists $u^{loc}_m\in H_{loc}^{1}\big(\overline{\Omega}_\eta\backslash\{z\}\big)\cap W_{loc}^{2,p}\big(\Omega_\eta\backslash\{z\}\big)$ such that
    \begin{gather} 
        Lu^{loc}_m=0,\quad \textrm{in}\quad\Omega_\eta\backslash\{z\},\label{eqn: sing eqn}\\
    u^{loc}_m=0,\quad\textrm{on}\quad\partial\Omega_\eta\quad\textrm{in the trace sense,}\label{eqn: extra req for sing}
    \end{gather}
    with
    \begin{align}
        \label{eqn: u sing}
        u^{loc}_m(x)=&\big(A^{-1}\big(z,a(z)\big)(x-z)\cdot(x-z)\big)^{\frac{2-n-m}{2}}m!\Big(A^{-1}_{nn}\big(z,a(z)\big)\Big)^{\frac{m}{2}}\nonumber\\
        &\times C_m^{\frac{n-2}{2}}\Bigg(\frac{A^{-1}_{n}\big(z,a(z)\big)(x-z)}{\Big(A^{-1}_{nn}\big(z,a(z)\big)\Big)^{\frac{1}{2}}\big(A^{-1}\big(z,a(z)\big)(x-z)\cdot(x-z)\big)^\frac{1}{2}}\Bigg)+v(x).
     \end{align}

     Moreover $v$ satisfies
    \begin{equation}
        \label{eqn: sing R1}
        |v(x)|+|x-z||Dv(x)|\leq C|x-z|^{2-n-m+\alpha},\quad \textrm{for any}\quad x\in B_\frac{\eta}{4}(z)\backslash\{z\},
    \end{equation}
    \begin{equation}
        \label{eqn: sing R2}
        \bigg(\int_{r<|x-z|<2r}|D^2v|^p\;dx\bigg)^{\frac{1}{p}}\leq Cr^{\frac{n}{p}-n+\alpha},\quad\textrm{for every}\quad r,\quad0<r<\eta/8.
    \end{equation}
    Here $\alpha$ is any number such that $0<\alpha<1-\frac{n}{p}$, and $C$ is a positive constant depending on $\alpha,n,p,R,\mathcal{E}_1,\mathcal{E}_2,k,\eta_0,\eta$ and $E$.
\end{theorem}
\begin{remark}
    Note that if $z\in U_\eta\backslash\Omega$, then $u_m^{loc}\in H^1(\Omega)$ and its trace satisfies $u_m^{loc}\big\vert_{\partial\Omega}\in H^\frac{1}{2}_{00}(\Sigma)$.
\end{remark}
\begin{proof}[Proof of Theorem \ref{sing thm}]
We consider a ball $B_R(z)$ with radius $R$ sufficiently large such that $\Omega_\eta\subset B_\frac{R}{2}(z)\Subset B_R(z)$. For a fixed non-negative integer $m$, we consider the singular solution $u_m$ on $B_R(z)$, having an isolated singularity at $x=z$, introduced in Theorem \ref{Jasons thm}, equation \eqref{u explicit} and define $\omega_m$ to be the solution to 
\begin{equation*}
    \begin{cases}
        L\omega_m=0, &\quad\textrm{in}\quad\Omega_\eta,\\
        \omega_m=-u_m, &\quad\textrm{on}\quad\partial\Omega_\eta.
    \end{cases}
\end{equation*}
By \eqref{u explicit}, we have
\begin{equation}
    \underset{\partial\Omega_\eta}{\sup}(|u_m|+|\nabla u_m|)\leq C_2,
\end{equation}
and by the trace theorem (see \cite[Theorem 3.12.12]{Ve}), we obtain
\begin{equation}
    \Vert\omega_m\Vert_{H^1(\Omega_\eta)}\leq C_3,
\end{equation}
where $C_2$ is a positive constant that depends on $\alpha,n,p,R,\mathcal{E}_1,\mathcal{E}_2,k$ and $E$ only, and $C_3$ depends on $\eta_0,\eta,n,m,R,L$ and $r_0$ only. By setting
\begin{equation}
    u^{loc}_m =u_m+\omega_{loc},\quad\textnormal{in}\quad\Omega_\eta,
\end{equation}
$u^{loc}_m$ satisfies \eqref{eqn: sing eqn} and \eqref{eqn: extra req for sing}. Moreover, by a standard interior regularity estimate \cite[Lemma 6.2.6]{Morreybook}, we have
\begin{equation}
    \Vert \omega_{loc}\Vert_{W^{2,p}(B_{\eta/4}(z))}\leq C,
\end{equation}
where $C$ depends only on $\eta_0,\eta,m,n,R,L$ and $r_0$. Setting $v=w+\omega_{loc}$, we have that $u^{loc}_m$ takes the form \eqref{eqn: u sing} and $v$ satisfies \eqref{eqn: sing R1}-\eqref{eqn: sing R2}.
\end{proof}


\section{Proof of Main Results}\label{sec4}
We start by recalling the following regularity result for $A(\cdot,a(\cdot))$, the proof of which can be found in \cite[Lemma 3.6]{A-G1}.


\begin{lemma}\label{A Sobolev Norm lemma}
    Under the hypotheses of Theorem \ref{thm 1}, we have
    \begin{equation}
        A\big(\cdot,a(\cdot)\big)\in W^{1,p}\big(\Omega,Sym_n\big),
    \end{equation}
    and furthermore,
    \begin{equation}
        \big\Vert A\big(\cdot,a(\cdot)\big)\big\Vert_{W^{1,p}(\Omega)}\leq CE\big(1+\Vert a\Vert_{W^{1,p}(\Omega)}\big),
    \end{equation}
    where $C$ is a positive constant depending only on $\lambda,\Omega,n$, and $p$.
\end{lemma}
Next, we recall that, since the boundary $\partial\Omega$ is of Lipschitz class, the normal unit vector field might not be defined on $\partial\Omega$. We therefore introduce a unitary vector field $\widetilde{\nu}$ locally defined near $\partial\Omega$ such that $\widetilde{\nu}$ is $C^\infty$ smooth, non-tangential to $\partial\Omega$ and points to the exterior of $\Omega$ (see \cite{A-G1} for a precise construction of $\widetilde{\nu}$).

We also recall that for $x^0\in\overline{\Sigma}_\eta$, the point $z_\tau=x^0+\tau\widetilde{\nu}\in\Omega_\eta\backslash\overline{\Omega}$ and satisfies
\begin{equation}
    \label{eqn: bound distance}
    C\tau\leq d(z_\tau,\partial\Omega)\leq\tau,\quad\textrm{for any}\quad\tau,\quad0\leq\tau\leq\tau_0,
\end{equation}
where $\tau_0$ and $C$ are positive constants depending on $L$ and $r_0$ only \cite[Lemma 2.2]{A-G1}.
\begin{remark} Several constants depending on the set of \textit{a-priori} data introduced in Definition \ref{apriori data} will appear in the paper. In order to simplify our notation, we will often denote by $C$ any of these constants, avoiding in most cases to point out their specific dependence on the \textit{a-priori} data.
\end{remark}


\begin{proof}[Proof of Theorem \ref{thm 1}]
We recall from \eqref{eqn: Aless I} that
\begin{equation*}
    \Big\langle\big(\Lambda_{a_1}^\Sigma-\Lambda_{a_2}^\Sigma\big)u_1,\overline{u_2}\Big\rangle=\int_\Omega\Big(A\big(x,a_1(x)\big)-A\big(x,a_2(x)\big)\Big)\nabla u_1\cdot \nabla u_2\; dx,
\end{equation*}
for any $u_i\in H^1(\Omega)$ solution to
\begin{equation}
    \label{eqn: eqn 1}
    \textrm{div}\big(A\big(x,a_i(x)\big)\nabla u_i\big)=0,\quad\textrm{for}\quad x\in\Omega, \quad\textrm{for}\quad i=1,2.
\end{equation}
Let $x^0\in\overline{\Sigma}_\eta$ be such that
\[(a_1-a_2)(x^0)=\Vert a_1-a_2\Vert_{L^\infty(\overline{\Sigma}_\eta)}.\]
Setting $z_{\tau} = x^0 + \tau\tilde\nu$, with $0<\tau\leq\min\left\{\tau_0,\: \frac{\eta}{8}\right\}$ and fixing $m=0$, we consider $u^{loc}_{i;\:0}\in W^{2,p}(\Omega)$ the singular solution \eqref{eqn: u sing} introduced in Theorem \ref{sing thm}, having a singularity at $z=z_\tau$, corresponding to the conductivity $A(\cdot,\:a_i(\cdot))$, for $i=1,2$. To ease our notation, we will simply denote
\begin{equation}
u^{loc}_{i;\:0} = u^{loc}_i,\qquad\text{for}\quad i=1,2,
\end{equation}
hence
\begin{equation}
    \begin{gathered}
    \label{eqn: Solns}
    u^{loc}_1(x)=\Big(A^{-1}\big(z_\tau,a_1(z_\tau)\big)(x-z_\tau)\cdot(x-z_\tau)\Big)^{\frac{2-n}{2}}+O(|x-z_\tau|^{2-n+\alpha}),\\
    u^{loc}_2(x)=\Big(A^{-1}\big(z_\tau,a_2(z_\tau)\big)(x-z_\tau)\cdot(x-z_\tau)\Big)^{\frac{2-n}{2}}+O(|x-z_\tau|^{2-n+\alpha}).
    \end{gathered}
\end{equation} 
By fixing $\rho>0$ and possibly reducing $\tau$, such that $0<\tau\leq\min\left\{\tau_0,\: \frac{\eta}{8},\:\frac{\rho}{2}\right\}$, we have that $B_\rho(z_\tau)\cap\Omega\neq\emptyset$ and $B_\rho(z_\tau)\cap\Omega\subset U_{\eta}$. From \eqref{eqn: Aless I}, we have
\begin{align}
    \label{eqn: I1}
        \bigg|\int_{B_\rho(z_\tau)\cap\Omega}&\Big(A\big(x,a_1(x)\big)-A\big(x,a_2(x)\big)\Big)\nabla u^{loc}_1\cdot \nabla u^{loc}_2\;dx\bigg|\nonumber\\
        \leq&\left|\int_{\Omega\backslash B_\rho(z_\tau)}\Big(A\big(x,a_1(x)\big)-A\big(x,a_2(x)\big)\Big)\nabla u^{loc}_1\cdot \nabla u^{loc}_2\;dx\right|\nonumber\\
        &+\big\Vert\Lambda_{a_1}^\Sigma-\Lambda_{a_2}^\Sigma\big\Vert_\ast\Vert u^{loc}_1\Vert_{H^{\frac{1}{2}}_{00}(\Sigma)}\Vert \overline{u^{loc}_2}\Vert_{H^{\frac{1}{2}}_{00}(\Sigma)}.
\end{align}
Recalling that for $i=1,2$, the real and imaginary parts of $A^{-1}$ satisfy \eqref{eqn: Re A inv bound}, and \eqref{eqn: AI inv bound1} or \eqref{eqn: AI inv bound2}, respectively, we have
\begin{equation}
    \label{eqn: ellip for A-1}
    C^{-1}|\xi|^2\leq|A^{-1}\big(x,a_i(x)\big)\xi\cdot\xi|\leq C|\xi|^2,\quad\textrm{for a.e.}\quad x\in\Omega,\quad\textrm{for every}\quad\xi\in\mathbb{R}^n.
\end{equation}
By combining \eqref{eqn: I1} together with \eqref{eqn: Solns},\eqref{eqn: ellip for A-1} and by Theorem \ref{sing thm},
\begin{align}
    \label{eqn: I2}
    \bigg|\int_{B_\rho(z_\tau)\cap\Omega}&\Big(A\big(x,a_1(x)\big)-A\big(x,a_2(x)\big)\Big)\nabla u^{loc}_1\cdot \nabla u^{loc}_2\; dx \bigg|\nonumber\\
    &\leq C\int_{\Omega\backslash B_\rho(z_\tau)}|x-z_\tau|^{2-2n}\; dx+\big\Vert\Lambda_{a_1}^\Sigma-\Lambda_{a_2}^\Sigma\big\Vert_\ast\Vert u^{loc}_1\Vert_{H^{\frac{1}{2}}_{00}(\Sigma)}\Vert \overline{u^{loc}_2}\Vert_{H^{\frac{1}{2}}_{00}(\Sigma)}.
    \end{align}
    The left-hand side of \eqref{eqn: I2} can be estimated from below by recalling that $A$ is H\"older continuous on $\overline{\Omega}$ with exponent of $\beta=1-n/p$ which leads to
\begin{align}
    \label{eqn: A Lips}
    \bigg|\int_{B_\rho(z_\tau)\cap\Omega}\Big(A\big(x,a_1&(x)\big)-A\big(x,a_2(x)\big)\Big)\nabla u^{loc}_1\cdot \nabla u^{loc}_2\; dx\bigg|\nonumber\\
\geq&\bigg|\int_{B_\rho(z_\tau)\cap\Omega}\Big(A\big(x^0,a_1(x^0)\big)-A\big(x^0,a_2(x^0)\big)\Big)\nabla u^{loc}_1\cdot \nabla u^{loc}_2\; dx\bigg|\nonumber\\&-C\int_{B_\rho(z_\tau)\cap\Omega}|x-x^0|^\beta|x-z_\tau|^{2-2n}\; dx.
\end{align}
Hence, combining \eqref{eqn: A Lips} together with \eqref{eqn: I2}, we obtain
\begin{align}
    \label{eqn: I3}
\bigg|\int_{B_\rho(z_\tau)\cap\Omega}\!&\Big(A\big(x^0,a_1(x^0)\big)\!-\!A\big(x^0,a_2(x^0)\big)\Big)\nabla u^{loc}_1(x)\!\cdot\! \nabla u^{loc}_2(x)\; dx\bigg|\nonumber\\
    \leq&\; C\biggl\{\int_{B_\rho(z_\tau)\cap\Omega}|x-x^0|^\beta|x-z_\tau|^{2-2n}\;dx+\int_{\Omega\backslash B_\rho(z_\tau)}|x-z_\tau|^{2-2n}\;dx\biggr\}\nonumber\\
    &+\big\Vert\Lambda_{a_1}^\Sigma-\Lambda_{a_2}^\Sigma\big\Vert_\ast\Vert u_1^{loc}\Vert_{H^{\frac{1}{2}}_{00}(\Sigma)}\Vert \overline{u_2^{loc}}\Vert_{H^{\frac{1}{2}}_{00}(\Sigma)}.
\end{align}
Recalling \eqref{eqn: ellip for A-1}, we can estimate the left-hand side of \eqref{eqn: I3} from below as 
\begin{align}
    \label{eqn: LHS lower}
\Bigg|\!\int_{B_\rho(z_\tau)\cap\Omega}\!\!\!\!\!\!\!\!&\frac{A^{-1}\!\big(\!z_\tau\, a_2(z_\tau)\!\big)\!\Big(\!A\big(\!x^0\!,a_1(x^0)\!\big)\!\!-\!\!A\big(\!x^0\!,a_2(x^0)\!\big)\!\!\Big)\!A^{-1}\!\big(\!z_\tau, a_1(z_\tau)\!\big)\!(x\!-\!z_\tau)\!\cdot\!(x\!-\!z_\tau)}{\big(\!A^{-1}\!\big(\!z_\tau, a_1(z_\tau)\!\big)\!(x\!-\!z_\tau)\!\cdot\!(x\!-\!z_\tau)\!\big)\!^{\frac{n}{2}}\big(\!A^{-1}\!\big(\!z_\tau, a_2(z_\tau)\!\big)\!(x\!-\!z_\tau)\!\cdot\!(x\!-\!z_\tau)\!\big)\!^{\frac{n}{2}}}dx\Bigg|\nonumber\\
    -&C\biggl\{\int_{B_\rho(z_\tau)\cap\Omega}|x-z_\tau|^{2-2n+\alpha}\;dx+\int_{B_\rho(z_\tau)\cap\Omega}|x-z_\tau|^{2-2n+2\alpha}\;dx\biggr\}\nonumber\\
    &\leq\bigg|\int_{B_\rho(z_\tau)\cap\Omega}\Big(A\big(x^0,a_1(x^0)\big)-A\big(x^0,a_2(x^0)\big)\Big)\nabla u^{loc}_1\cdot \nabla u^{loc}_2\; dx\bigg|.
\end{align}
\eqref{eqn: LHS lower} together with \eqref{eqn: I3} leads to
\begin{align}
    \label{eqn: I4}   
\Bigg|\!\int_{B_\rho(z_\tau)\cap\Omega}\!\!\!\!\!\!\!\!&\frac{A^{-1}\!\big(\!z_\tau\, a_2(z_\tau)\!\big)\!\Big(\!A\big(\!x^0\!,a_1(x^0)\!\big)\!\!-\!\!A\big(\!x^0\!,a_2(x^0)\!\big)\!\!\Big)\!A^{-1}\!\big(\!z_\tau, a_1(z_\tau)\!\big)\!(x\!-\!z_\tau)\!\cdot\!(x\!-\!z_\tau)}{\big(\!A^{-1}\!\big(\!z_\tau, a_1(z_\tau)\!\big)\!(x\!-\!z_\tau)\!\cdot\!(x\!-\!z_\tau)\!\big)\!^{\frac{n}{2}}\big(\!A^{-1}\!\big(\!z_\tau, a_2(z_\tau)\!\big)\!(x\!-\!z_\tau)\!\cdot\!(x\!-\!z_\tau)\!\big)\!^{\frac{n}{2}}}dx\Bigg|\nonumber\\
    \leq& C\biggl\{\int_{B_\rho(z_\tau)\cap\Omega}\!\!\!|x-z_\tau|^{2-2n+\alpha}\;dx+\int_{B_\rho(z_\tau)\cap\Omega}\!\!\!\!|x-x^0|^\beta|x-z_\tau|^{2-2n}\;dx\nonumber\\
     &+\int_{\Omega\backslash B_\rho(z_\tau)}\!\!\!\!\!|x-z_\tau|^{2-2n}\;dx\biggr\}+\big\Vert\Lambda_{a_1}^\Sigma-\Lambda_{a_2}^\Sigma\big\Vert_\ast\Vert u^{loc}_1\Vert_{H^{\frac{1}{2}}_{00}(\Sigma)}\Vert \overline{u^{loc}_2}\Vert_{H^{\frac{1}{2}}_{00}(\Sigma)}.  
\end{align}
Next, we estimate from below the left-hand side of \eqref{eqn: I4} as follows,
\begin{align}
    \label{eqn: A-1 Lips}
\Bigg|\!\int_{B_\rho(\!z_\tau\!)\cap\Omega}\!\!\!\!\!\!\!\!&\frac{A^{-1}\!\big(\!z_\tau\, a_2(z_\tau)\!\big)\!\Big(\!A\big(\!x^0\!,a_1(x^0)\!\big)\!\!-\!\!A\big(\!x^0\!,a_2(x^0)\!\big)\!\!\Big)\!A^{-1}\!\big(\!z_\tau, a_1(z_\tau)\!\big)\!(x\!-\!z_\tau)\!\cdot\!(x\!-\!z_\tau)}{\big(\!A^{-1}\!\big(\!z_\tau, a_1(z_\tau)\!\big)\!(x\!-\!z_\tau)\!\cdot\!(x\!-\!z_\tau)\!\big)\!^{\frac{n}{2}}\big(\!A^{-1}\!\big(\!z_\tau, a_2(z_\tau)\!\big)\!(x\!-\!z_\tau)\!\cdot\!(x\!-\!z_\tau)\!\big)\!^{\frac{n}{2}}}dx\Bigg|\nonumber\\
    \geq&\Bigg|\!\int_{B_\rho(\!z_\tau\!)\cap\Omega}\!\!\frac{\Big(A^{-1}\big(x^0,a_2(x^0)\big)\!-\!A^{-1}\big(x^0,a_1(x^0)\big)\Big)(x-z_\tau)\!\cdot\!(x-z_\tau)}{\big(\!A^{-1}\!\big(\!z_\tau, a_1(z_\tau)\!\big)\!(x\!-\!z_\tau)\!\cdot\!(x\!-\!z_\tau)\!\big)\!^{\frac{n}{2}}\!\big(\!A^{-1}\!\big(\!z_\tau, a_2(z_\tau)\!\big)\!(x\!-\!z_\tau)\!\cdot\!(x\!-\!z_\tau)\!\big)\!^{\frac{n}{2}}}\;\!dx\Bigg|\nonumber\\
    &\qquad-C\int_{B_\rho(z_\tau)\cap\Omega}|x-z_\tau|^{2-2n+\beta}\;dx.
\end{align}
Combining \eqref{eqn: A-1 Lips} with \eqref{eqn: I4} leads to
\begin{align}
    \label{eqn: I5}
\Bigg|\int_{B_\rho(z_\tau)\cap\Omega}&\!\frac{\Big(A^{-1}\big(x^0,a_2(x^0)\big)-A^{-1}\big(x^0,a_1(x^0)\big)\Big)(x-z_\tau)\cdot(x-z_\tau)}{\big(\!A^{-1}\!\big(z_\tau, a_1(z_\tau)\big)\!(x\!-\!z_\tau)\!\cdot\!(x\!-\!z_\tau)\!\big)\!^{\frac{n}{2}}\!\big(\!A^{-1}\!\big(z_\tau, a_2(z_\tau)\big)\!(x\!-\!z_\tau)\!\cdot\!(x\!-\!z_\tau)\!\big)\!^{\frac{n}{2}}}\;dx\Bigg|\nonumber\\
    \leq&C\Biggl\{\int_{\Omega\backslash B_\rho(z_\tau)}|x-z_\tau|^{2-2n}\;dx+\int_{B_\rho(z_\tau)\cap\Omega}|x-z_\tau|^{2-2n+\beta}\;dx\nonumber\\
    &+\int_{B_\rho(z_\tau)\cap\Omega}|x-z_\tau|^{2-2n+\alpha}\;dx+\int_{B_\rho(z_\tau)\cap\Omega}|x-x^0|^\beta|x-z_\tau|^{2-2n}\;dx\Biggr\}\nonumber\\
    &+\big\Vert\Lambda_{a_1}^\Sigma-\Lambda_{a_2}^\Sigma\big\Vert_\ast\Vert u_1^{loc}\Vert_{H^{\frac{1}{2}}_{00}(\Sigma)}\Vert \overline{u_2^{loc}}\Vert_{H^{\frac{1}{2}}_{00}(\Sigma)}.
\end{align}
The integrand appearing on the left-hand side of \eqref{eqn: I5} can be expressed as 
\begin{equation}
    \label{eqn: complex frac}
    \frac{F(x)}{\big|A^{-1}\big(z_\tau, a_1(z_\tau)\big)(x-z_\tau)\cdot(x-z_\tau)\big|^n\big|A^{-1}\big(z_\tau, a_2(z_\tau)\big)(x-z_\tau)\cdot(x-z_\tau)\big|^n},
\end{equation}
where the complex-valued function $F$ is defined by
\begin{align}
    \label{eqn: F}
     F(x)=&\Big(A^{-1}\big(x^0,a_2(x^0)\big)-A^{-1}\big(x^0,a_1(x^0)\big)\Big)(x-z_\tau)\!\cdot\!(x-z_\tau)\nonumber\\
    &\times\Big(\overline{A^{-1}\big(z_\tau, a_1(z_\tau)\big)}(x-z_\tau)\!\cdot\!(x-z_\tau)\Big)^{\frac{n}{2}}\!\Big(\overline{A^{-1}\big(z_\tau, a_2(z_\tau)\big)}(x-z_\tau)\!\cdot\!(x-z_\tau)\Big)^{\frac{n}{2}}\!.
\end{align}
The choice of $k$ in \eqref{eqn: k1} implies
\begin{equation}
     \label{eqn: F assumptions}
    |\Im F(x)|\leq|\Re F(x)|,\qquad \Re F(x)>0.
\end{equation}
Using \eqref{eqn: F assumptions}, the left-hand side of inequality \eqref{eqn: I5} can be estimated from below as
\begin{align}
    \label{eqn: Using F ass}
    \Bigg|\!&\int_{B_\rho(\!z_\tau\!)\cap\Omega}\!\!\frac{F(x)}{\big|A^{-1}\!\big(z_\tau, a_1(z_\tau)\big)\!(x\!-\!z_\tau)\!\cdot\!(x\!-\!z_\tau)\!\big|^n \big|A^{-1}\!\big(z_\tau, a_2(z_\tau)\big)\!(x\!-\!z_\tau)\!\cdot\!(x\!-\!z_\tau)\!\big|^n}\;dx\Bigg|\nonumber\\
    &\geq\! \Re\!\Biggl\{\!\int_{B_\rho(\!z_\tau\!)\cap\Omega}\!\!\frac{F(x)}{\big|\!A^{-1}\!\big(\!z_\tau, a_1(z_\tau)\!\big)\!(x\!-\!z_\tau)\!\cdot\!(x\!-\!z_\tau)\!\big|^n\! \big|\!A^{-1}\!\big(\!z_\tau, a_2(z_\tau)\!\big)\!(x\!-\!z_\tau)\!\cdot\!(x\!-\!z_\tau)\!\big|^n}\;dx\!\!\Biggr\}\nonumber\\
    &\geq\!\frac{1}{\sqrt{2}}\!\int_{B_\rho(\!z_\tau\!)\cap\Omega}\!\!\frac{|F(x)|}{\big|\!A^{-1}\!\big(\!z_\tau, a_1(z_\tau)\!\big)\!(x\!-\!z_\tau)\!\cdot\!(x\!-\!z_\tau)\!\big|^n\! \big|\!A^{-1}\!\big(\!z_\tau, a_2(z_\tau)\!\big)\!(x\!-\!z_\tau)\!\cdot\!(x\!-\!z_\tau)\!\big|^n}\;dx.
\end{align}
Using \eqref{eqn: ellip for A-1} and the monotonicity condition \eqref{eqn: mono}, we compute
\begin{align}
    \label{eqn: F norm}
    |F(x)|\geq&\Re\biggl\{\int_{a_2(x^0)}^{a_1(x^0)}D_tA(x^0,t)A(x^0,t)^{-1}(x-z_\tau)\!\cdot\! A(x^0,t)^{-1}(x-z_\tau)\;dt\biggr\}\nonumber\\
    &\times\Big|\overline{A^{-1}\big(z_\tau, a_1(z_\tau)\big)}(x-z_\tau)\!\cdot\!(x-z_\tau)\Big|^{\frac{n}{2}}\Big|\overline{A^{-1}\big(z_\tau, a_2(z_\tau)\big)}(x-z_\tau)\!\cdot\!(x-z_\tau)\Big|^{\frac{n}{2}} \nonumber\\
    \geq& C\big(a_1(x^0)-a_2(x^0)\big)|x-z_\tau|^{2+2n}.
\end{align}
By combining \eqref{eqn: Using F ass} and \eqref{eqn: F norm} with \eqref{eqn: I5}, we obtain
\begin{align}
    \label{eqn: I6}
    \big(a_1(x^0)-a_2(x^0)&\big)\int_{B_\rho(z_\tau)\cap\Omega}|x-z_\tau|^{2-2n}\;dx\nonumber\\
    \leq& C\Biggl\{\int_{\Omega\backslash B_\rho(z_\tau)}|x-z_\tau|^{2-2n}\;dx+\int_{B_\rho(z_\tau)\cap\Omega}|x-z_\tau|^{2-2n+\beta}\;dx\nonumber\\
     &+\int_{B_\rho(z_\tau)\cap\Omega}|x-z_\tau|^{2-2n+\alpha}\;dx+\int_{B_\rho(z_\tau)\cap\Omega}|x-x^0|^\beta|x-z_\tau|^{2-2n}\;dx\nonumber\\
     &+\big\Vert\Lambda_{a_1}^\Sigma-\Lambda_{a_2}^\Sigma\big\Vert_\ast\Vert u^{loc}_1\Vert_{H^{\frac{1}{2}}_{00}(\Sigma)}\Vert \overline{u^{loc}_2}\Vert_{H^{\frac{1}{2}}_{00}(\Sigma)}\Biggr\}.
\end{align}
By estimating the integrals in \eqref{eqn: I6} and the $H^{\frac{1}{2}}_{00}(\Sigma)$ norms of $u^{loc}_1$ and $\overline{u^{loc}_2}$, we obtain
\begin{equation}
    \label{eqn: Final 1}
    \Vert a_1-a_2\Vert_{L^\infty(\overline{\Sigma}_\eta)}\tau^{2-n}\leq C\Big\{C\tau^{2-n+\beta}+C\tau^{2-n+\alpha}+C+\tau^{2-n}\big\Vert\Lambda_{a_1}^\Sigma-\Lambda_{a_2}^\Sigma\big\Vert_\ast\Big\},
\end{equation}
Therefore,
\begin{equation}
    \label{eqn: Final 2}
    \Vert a_1-a_2\Vert_{L^\infty(\overline{\Sigma}_\eta)}\leq C\Big\{\mu(\tau)+\big\Vert\Lambda_{a_1}^\Sigma-\Lambda_{a_2}^\Sigma\big\Vert_\ast\Big\},
\end{equation}
where $\mu(\tau)\rightarrow0$ as $\tau\rightarrow0$, hence, from \eqref{eqn: Final 2} we obtain 
\begin{equation}
    \label{eqn: pre result}
    \Vert a_1-a_2\Vert_{L^\infty(\overline{\Sigma}_\eta)}\leq C\big\Vert\Lambda_{a_1}^\Sigma-\Lambda_{a_2}^\Sigma\big\Vert_\ast.
\end{equation}
Recalling that, for a.e. $x\in\Omega$, the function
\[t\longrightarrow A(x,t)\]
is absolutely continuous on $[\lambda^{-1},\lambda]$, we have
\begin{equation*}
    A\big(x,a_1(x)\big)=A\big(x,\lambda\big)-\int_{a_1(x)}^\lambda D_tA(x,t)\;dt,
\end{equation*}
for a.e. $x\in\Omega$ (see \cite[Lemma 3.1.1]{Morreybook}). Therefore, for every $x\in\overline{\Omega}$, we have
\begin{align}\label{A integral a}
    \big|A\big(x,a_1(x)\big)-A\big(x,a_2(x)\big)\big|= &\Big|\int_{a_2(x)}^{a_1(x)}D_tA(x,t)\;dt\Big|\nonumber\\
    \leq&\int_{a_2(x)}^{a_1(x)}\textrm{Sup}_{t,x}|D_tA(x,t)|\;dt\nonumber\\
    \leq&C|a_1(x)-a_2(x)|.
\end{align} 
Taking the $L^\infty$-norm on both sides of \eqref{A integral a}, we obtain
\begin{equation}
    \label{eqn: pre 2}
    \Vert A(x,a_1(x))-A(x,a_2(x))\Vert_{L^\infty(\overline{\Sigma}_\eta)}\leq C\Vert a_1(x)-a_2(x)\Vert_{L^\infty(\overline{\Sigma}_\eta)}.
\end{equation}
Combining \eqref{eqn: pre 2} with \eqref{eqn: pre result} concludes the proof.
\end{proof}
Next, for the proof of Theorem \ref{der thm} we also need to control $Du^{loc}_m$ near the singularity $z$ from below. This is achieved in the following Lemma. Without loss of generality, we set $z=0$. 


\begin{lemma}\label{Lem 4.1}
    Let the hypothesis of Theorem \ref{sing thm} be satisfied. Then, for any $m=0,1,2,..$, the singular solution $u^{loc}_m$ in \eqref{eqn: u sing} having an isolated singularity at $z=0$ also satisfies
    \begin{equation}\label{gradient from below}
        |Du^{loc}_m(x)|>C|x|^{1-(n+m)},\quad\textrm{for every}\quad x\in\Omega,\quad0<|x|\leq r_1,
    \end{equation}
    where $C$ and $r_1$ are positive constants depending on $n,p,m,\mathcal{E}_1,\mathcal{E}_2,k, \Omega$ and $E$.
\end{lemma}


\begin{proof}[Proof of Lemma \ref{Lem 4.1}]
      For a fixed $m$, the singular solution of Theorem \ref{sing thm} having an isolated singularity at $z=0$ takes the form
    \begin{align}\label{eqn: um+v}
        u^{loc}_m(x)=u_m(x)+v(x)=&\big(A^{-1}\big(0,a(0)\big)x\cdot x\big)^{\frac{2-n-m}{2}}m!\Big(A^{-1}_{nn}\big(0,a(0)\big)\Big)^{\frac{m}{2}}\nonumber\\
        &\times C_m^{\frac{n-2}{2}}\Bigg(\frac{A^{-1}_{n}\big(0,a(0)\big)x}{\Big(A^{-1}_{nn}\big(0,a(0)\big)\Big)^{\frac{1}{2}}\Big(A^{-1}\big(0,a(0)\big)x\cdot x\Big)^\frac{1}{2}}\Bigg)+v(x).\nonumber\\
     \end{align}
     To prove \eqref{gradient from below}, we will show that
     \begin{equation}
         \big|Du_m(x)\big|>C|x|^{1-(n+m)},\qquad\textnormal{for}\quad x\in\Omega\cap B_{r_1}(0).
     \end{equation}
     It is a straightforward calculation to show that
     \begin{align}
         Du_m(x)=&m!\Big(A^{-1}_{nn}\big(0,a(0)\big)\Big)^{\frac{m}{2}}\bigg\{\Big(A^{-1}\big(0,a(0)\big)x\cdot x\Big)^\frac{2-n-m}{2}\frac{dC_m^{\frac{n-2}{2}}(t)}{dt}Dt\nonumber\\
        &+(2-n-m)\Big(A^{-1}\big(0,a(0)\big)x\Big)\Big(A^{-1}\big(0,a(0)\big)x\cdot x\Big)^\frac{-n-m}{2}C_m^{\frac{n-2}{2}}(t)\bigg\},
     \end{align}
     where 
     \begin{equation}\label{t}
        t=\frac{A^{-1}_{n}\big(0,a(0)\big)x}{\Big(A^{-1}_{nn}\big(0,a(0)\big)\Big)^{\frac{1}{2}}\Big(A^{-1}\big(0,a(0)\big)x\cdot x\Big)^\frac{1}{2}}.
    \end{equation}
    We have
    \begin{align}
        |Du_m(x)|^2=&(m!)^2\Big|A^{-1}_{nn}\big(0,a(0)\big)\Big|^m\bigg|\Big(A^{-1}\big(0,a(0)\big)x\cdot x\Big)^{\frac{2-n-m}{2}}\frac{dC_m^{\frac{n-2}{2}}(t)}{dt}Dt\nonumber\\
        &+(2-n-m)\Big(A^{-1}\big(0,a(0)\big)x\cdot x\Big)^{\frac{-n-m}{2}}\Big(A^{-1}\big(0,a(0)\big)x\Big)C_m^{\frac{n-2}{2}}(t)\bigg|^2.
    \end{align}
   Setting
    \begin{align}\label{eqn: lemma F def}
        h(x) &=|x|^{2n+2m-2}\bigg|\Big(A^{-1}\big(0,a(0)\big)x\cdot x\Big)^{\frac{2-n-m}{2}}\frac{dC_m^{\frac{n-2}{2}}(t)}{dt}Dt\nonumber\\
        &+(2-n-m)\Big(A^{-1}\big(0,a(0)\big)x\cdot x\Big)^{\frac{-n-m}{2}}\Big(A^{-1}\big(0,a(0)\big)x\Big)C_m^{\frac{n-2}{2}}(t)\bigg|^2,
    \end{align}
    we have that $h(cx)=h(x)$, for any $c\in\mathbb{R}\setminus\{0\}$. Hence $h$ is determined by its restriction on the unit sphere $S^{n-1}$. Due to the continuity of $h$ on $S^{n-1}$, $h$ has a minimum on $S^{n-1}$. Such a minimum is positive since $h(x)\neq0$ for all $x\in S^{n-1}$. In fact, on one hand, $C_m^{\frac{n-2}{2}}(t)$ satisfies
     \begin{equation}
         (1-t^2)y''-(n-1)ty'+m(m+n-2)y=0,
     \end{equation}
     for all $t\in\mathbb{C}$ (see \cite[Section 3.15.2]{highertransfns}). Since $C_m^{\frac{n-2}{2}}(\pm1)\neq0$, by the Cauchy uniqueness theorem, $C_m^{\frac{n-2}{2}}(t)$ and $\frac{dC_m^{\frac{n-2}{2}}(t)}{dt}$ cannot simultaneously vanish. Also, the quantity
     \begin{equation}
         Dt(x)=\frac{\Big(A^{-1}\big(0,a(0)\big)x\cdot x\Big)\Big(A^{-1}_{n}\big(0,a(0)\big)\Big)\!-\!\Big(A^{-1}_{n}\big(0,a(0)\big)x\Big)\Big(A^{-1}\big(0,a(0)\big)x\Big)}{\Big(A^{-1}_{nn}\big(0,a(0)\big)\Big)^{\frac{1}{2}}\Big(A^{-1}\big(0,a(0)\big)x\cdot x\Big)^\frac{3}{2}},
     \end{equation}
     vanishes only when $x=e_n\in S^{n-1}$ (here $e_n$ denotes the $n^{th}$ vector in the canonical basis for $\mathbb{R}^n$). In this case we have that $t$ defined in \eqref{t} is equal to $1$ and $C_m^{\frac{n-2}{2}}(1)\neq 0$. Consequently, 
     \begin{eqnarray}
     & & \Big(A^{-1}\big(0,a(0)\big)x\cdot x\Big)^{\frac{2-n-m}{2}}\frac{dC_m^{\frac{n-2}{2}}(t)}{dt}Dt;\label{h term 1}\\
    & & (2-n-m)\Big(A^{-1}\big(0,a(0)\big)x\cdot x\Big)^{\frac{-n-m}{2}}C_m^{\frac{n-2}{2}}(t)\Big(A^{-1}\big(0,a(0)\big)x\Big)\label{h term 2}
    \end{eqnarray}
     cannot simultaneously vanish. On the other hand the vector fields in \eqref{h term 1} and \eqref{h term 2} cannot be opposite as $Dt(x)$ and $A^{-1}\big(0,a(0)\big)x$ cannot be parallel vector fields, for $x\in S^{n-1}$. In fact, for $x\in S^{n-1}$, $Dt(x)\cdot x=0$, whereas $A^{-1}\big(0,a(0)\big)x\cdot x \neq 0$, due to the ellipticity condition \eqref{eqn: AR assump}. This implies that
     \begin{equation}
         h(x)\geq C>0,\quad\textrm{for all}\quad x\in S^{n-1},
     \end{equation}
concluding the proof.
\end{proof}
Next, we proceed to proving Theorem \ref{der thm}.

\begin{proof}[Proof of Theorem \ref{der thm}]
We start again by recalling \eqref{eqn: Aless I},
\begin{equation}\label{Ale bis}
    \Big\langle\big(\Lambda_{a_1}^\Sigma-\Lambda_{a_2}^\Sigma\big)u_1,\overline{u_2}\Big\rangle=\int_\Omega\Big(A\big(x,a_1(x)\big)-A\big(x,a_2(x)\big)\Big)\nabla u_1\cdot \nabla u_2\; dx,
\end{equation}
for any $u_i\in H^1(\Omega)$ solution to
\begin{equation}
    \textrm{div}(A\big(x,a_i(x)\big)\nabla u_i)=0,\quad\textrm{for}\quad x\in\Omega, \quad\textrm{for}\quad i=1,2.
\end{equation}
We fix an integer $h\geq 1$, we set $x^0\in\overline{\Sigma}_\eta$ such that
\[(-1)^h\frac{\partial^h}{\partial\widetilde{\nu}^h}(a_1-a_2)(x^0)=\bigg\Vert \frac{\partial^h}{\partial\widetilde{\nu}^h}(a_1-a_2)\bigg\Vert_{L^\infty(\overline{\Sigma}_\eta)}\]
and let $z_\tau=x^0+\tau\widetilde{\nu}$, with $0<\tau\leq\tau_0$, where $\tau_0$ is the number fixed in \eqref{eqn: bound distance}. For a fixed integer $m>0$, we consider the singular solutions $u^{loc}_{m,\:i}\in W^{2,p}(\Omega)$ to \eqref{eqn: eqn 1} of Theorem \ref{sing thm}, having an isolated singularity at $z_\tau$, for $i=1,2$. By simply denoting such solution by $u^{loc}_i$, this is given by
\begin{equation}
    u^{loc}_i(x)\!=\!C_i\!\Big(\!A^{-1}\big(z_\tau,a_i(z_\tau)\big)(x-z_\tau)\cdot(x-z_\tau)\Big)\!^{\frac{2-n-m}{2}}\!C_m^{\frac{n-2}{2}}\!(\widetilde{z}_1)+O(|x-z_\tau|^{2-n-m+\alpha}),
   \end{equation}
where $C_i=m!\Big(A^{-1}_{nn}\big(z_\tau,a_i(z_\tau)\big)\Big)^{\frac{m}{2}}$, and
\begin{equation}
    \widetilde{z_i}=\frac{A^{-1}_{n}\big(z_\tau,a_i(z_\tau)\big)(x-z_\tau)}{\Big(A^{-1}_{nn}\big(z_\tau,a_i(z_\tau)\big)\Big)^{\frac{1}{2}}\Big(A^{-1}\big(z_\tau,a_i(z_\tau)\big)(x-z_\tau)\cdot(x-z_\tau)\Big)^\frac{1}{2}},
\end{equation}
for $i=1,2$. Next, we prove that for $j\leq h$, 
\begin{equation}
    \label{eqn: induction aim}
    \bigg\Vert \frac{\partial^j}{\partial\widetilde{\nu}^j}(a_1-a_2)\bigg\Vert_{L^\infty(\overline{\Sigma}_\eta)}\leq C\big\Vert\Lambda_{a_1}^\Sigma-\Lambda_{a_2}^\Sigma\big\Vert_\ast^{\delta_j},
\end{equation} 
by induction on $j$. For $j=0$, \eqref{eqn: induction aim} is given by \eqref{eqn: pre result} in the proof of Theorem \ref{thm 1}. Now assuming \eqref{eqn: induction aim} holds true for all $j$, $j\leq h-1$, we will prove it holds true for $j=h$ too. Arguing as in the proof of Theorem \ref{thm 1}, we fix $\rho>0$ and possibly reducing $\tau$, such that $0<\tau\leq\min\left\{\tau_0,\: \frac{\eta}{8}.\:\frac{\rho}{2}\right\}$, we have that $\Omega\cap B_\rho(z_\tau)\neq\emptyset$ and such that $\Omega\cap B_\rho(z_\tau)\subset U_\eta$. From \eqref{Ale bis}, we obtain
\begin{align}\label{I1 der}
    \big\Vert\Lambda_{a_1}^\Sigma-\Lambda_{a_2}^\Sigma\big\Vert_\ast &\Vert u^{loc}_1\Vert_{H^{\frac{1}{2}}_{00}(\Sigma)}\Vert \overline{u^{loc}_2}\Vert_{H^{\frac{1}{2}}_{00}(\Sigma)}\nonumber\\
    \geq &\Big|\int_{B_\rho(z_\tau)\cap\Omega}\Big(A\big(x,a_1(x)\big)-A\big(x,a_2(x)\big)\Big)\nabla u^{loc}_1\cdot \nabla u^{loc}_2\,dx\Big|\nonumber\\
    -&\int_{\Omega\backslash B_\rho(z_\tau)}\left|A\big(x,a_1(x)\big)-A\big(x,a_2(x)\big)\right| |\nabla u^{loc}_1| |\nabla u^{loc}_2|\,dx,  
\end{align}
hence, noticing that the integral on the right-hand side of \eqref{I1 der} over $\Omega\backslash B_\rho(z_\tau)$ is bounded above by a positive constant $C$ depending on the \textit{a-priori} data, we obtain
\begin{align}
    \Vert\Lambda_{a_1}^\Sigma-\Lambda_{a_2}^\Sigma\Vert_\ast &\Vert u^{loc}_1\Vert_{H^{\frac{1}{2}}_{00}(\Sigma)}\Vert \overline{u^{loc}_2}\Vert_{H^{\frac{1}{2}}_{00}(\Sigma)}\nonumber\\
        \geq&\Big|\int_{B_\rho(z_\tau)\cap\Omega}\Big(A\big(x,a_1(x)\big)-A\big(x,a_2(x)\big)\Big)\nabla u^{loc}_1\cdot \nabla u^{loc}_2\,dx\Big|-C.     \end{align}
By assumption \eqref{eqn: A C norm} and the Lagrange Theorem, for any $x\in \Omega\cap B_\rho(z_\tau)$, there exists $s(x)\in(0,1)$ such that
\begin{equation}
    A\big(x,a_1(x)\big)-A\big(x,a_2(x)\big)=\frac{\partial A(x,t)}{\partial t}\Big|_{t=c(x)}\big(a_1(x)-a_2(x)\big),
\end{equation}
where $c(x)=a_1(x)+s(x)\big(a_2(x)-a_1(x)\big)$. Without loss of generality we assume that $a_1(x)-a_2(x)>0$, for a.e. $x\in \Omega\cap B_\rho(z_\tau)$. Therefore, we have
\begin{align}
    \label{eqn: I2.1}
    \big\Vert\Lambda_{a_1}^\Sigma-\Lambda_{a_2}^\Sigma&\big\Vert_\ast\Vert u^{loc}_1\Vert_{H^{\frac{1}{2}}_{00}(\Sigma)}\Vert \overline{u^{loc}_2}\Vert_{H^{\frac{1}{2}}_{00}(\Sigma)}\nonumber\\
        \geq&\Big|\int_{B_\rho(z_\tau)\cap\Omega}\big(a_1(x)-a_2(x)\big)\frac{\partial A(x,t)}{\partial t}\Big|_{t=c(x)}\nabla u^{loc}_1\cdot \nabla u^{loc}_2\,dx\Big|-C\nonumber\\
        \geq& \int_{B_\rho(z_\tau)\cap\Omega}\big(a_1(x)-a_2(x)\big)\Re\bigg\{\frac{\partial A(x,t)}{\partial t}\Big|_{t=c(x)}\nabla u^{loc}_1\cdot \nabla u^{loc}_2\bigg\}\,dx-C,
\end{align}
where $\Re(z)$ denotes the real part of $z\in\mathbb{C}$. Possibly reducing $\tau$, such that $0<\tau\leq\min\left\{\tau_0,\: \frac{\eta}{8}.\:\frac{\rho}{2}, \frac{r_1}{2}\right\}$, where $r_1$ is the positive constant introduced in Lemma \ref{Lem 4.1}, we estimate the remaining integral appearing on the right hand side of \eqref{eqn: I2.1} from below, by showing that
\begin{equation}
    \label{eqn: re bound}
    \Re\bigg\{\frac{\partial A(x,t)}{\partial t}\Big|_{t=c(x)}\nabla u_1^{loc}\cdot \nabla u_2^{loc}\bigg\}\geq C|x-z_\tau|^{2-2n-2m},\quad\textrm{ for a.e.}\quad x\in\Omega\cap B_\rho(z_\tau).
\end{equation}
In fact, we observe that by \eqref{eqn: mono} we have
\begin{align}
    \Re\bigg\{\frac{\partial A(x,t)}{\partial t}&\Big|_{t=c(x)}\nabla u_1^{loc}\cdot \nabla u_2^{loc}\bigg\}\nonumber\\
    =& \Re\bigg\{\frac{\partial A(x,t)}{\partial t}\Big|_{t=c(x)}\nabla u_1^{loc}\!\cdot\! \nabla u_1^{loc}+\frac{\partial A(x,t)}{\partial t}\Big|_{t=c(x)}\nabla u_1^{loc}\!\cdot\! (\nabla u_2^{loc}\!-\!\nabla u_1^{loc})\bigg\}\nonumber\\
    \geq& \Re\bigg\{\frac{\partial A(x,t)}{\partial t}\Big|_{t=c(x)}\nabla u_1^{loc}\!\cdot\! \nabla u_1^{loc}\bigg\}-\bigg|\frac{\partial A(x,t)}{\partial t}\Big|_{t=c(x)}\nabla u_1^{loc}\!\cdot\! (\nabla u_2^{loc}\!-\!\nabla u_1^{loc})\bigg|\nonumber\\
    \geq & E^{-1}|\nabla u_1^{loc}|^2-C|\nabla u_1^{loc}||\nabla u_2^{loc}-\nabla u_1^{loc}|.
\end{align}
It is a straightforward calculation to show that
\begin{equation}
    |\nabla u_i^{loc}|\leq C|x-z_\tau|^{1-n-m},\quad \textrm{for any}\quad x\in \Omega.
\end{equation}
By Lemma \ref{Lem 4.1}, for any $x\in B_{\rho}(z_{\tau})\cap\Omega$, we have
\begin{equation}
    \Re\bigg\{\frac{\partial A(x,t)}{\partial t}\Big|_{t=c(x)}\nabla u_1^{loc}\cdot \nabla u_2^{loc}\bigg\}\geq C|x-z_\tau|^{1-n-m}\Big\{|x-z_\tau|^{1-n-m}-|\nabla u_2^{loc}-\nabla u_1^{loc}|\Big\}
\end{equation}
and \eqref{eqn: u sing}, \eqref{eqn: sing R1} lead to
\begin{align}
    |\nabla u_1^{loc}-\nabla u_2^{loc}|\leq & C\Big\{ |a_1(x^0)-a_2(x^0)|^\beta|x-z_\tau|^{1-n-m}+\tau^\beta|x-z_\tau|^{1-n-m}\nonumber\\
    &+|x-z_\tau|^{1-n-m+\alpha}\Big\}.
\end{align}
Recalling $|x-z_\tau|>C\tau$, for any $x\in\Omega\cap B_\rho(z_\tau)$, and $\alpha<\beta$, we obtain 
\begin{equation}
    |\nabla u_1^{loc}-\nabla u_2^{loc}|\leq C\Big\{|a_1(x^0)-a_2(x^0)|^\beta|x-z_\tau|^{1-n-m}+|x-z_\tau|^{1-n-m+\alpha}\Big\}.
\end{equation}
Hence, for a.e. $x\in B_\rho(z_\tau)\cap\Omega$, we have
\begin{equation}
    \Re\bigg\{\!\frac{\partial A(x,t)}{\partial t}\Big|_{t=c(x)}\nabla u_1^{loc}\cdot \nabla u_2^{loc}\!\bigg\}\geq C|x-z_\tau|^{2-2n-2m}\Big\{\!1-|a_1(x^0)-a_2(x^0)|^\beta-|x-z_\tau|^\alpha\!\Big\}.
\end{equation}
From \eqref{eqn: pre result} in Theorem \ref{thm 1}, we have
\begin{equation}
    |a_1(x^0)-a_2(x^0)|\leq C \big\Vert\Lambda_{a_1}^\Sigma-\Lambda_{a_2}^\Sigma\big\Vert_\ast,
\end{equation}
and without loss of generality, we can assume that
\begin{equation}
    \big\Vert\Lambda_{a_1}^\Sigma-\Lambda_{a_2}^\Sigma\big\Vert_\ast\leq \frac{1}{2C},
\end{equation}
which leads to 
\begin{equation}
    \Re\bigg\{\frac{\partial A(x,t)}{\partial t}\Big|_{t=c(x)}\nabla u_1^{loc}\cdot \nabla u_2^{loc}\bigg\}\geq C|x-z_\tau|^{2-2n-2m}\bigg\{\frac{1}{2}-|x-z_\tau|^\alpha\bigg\}.
\end{equation}
By reducing $\rho$ so that $|x-z_\tau|^\alpha<\frac{1}{2}$, we obtain \eqref{eqn: re bound}, which, combined with \eqref{eqn: I2.1} leads to
\begin{align}
    \label{eqn: I2.2}
    \big\Vert\Lambda_{a_1}^\Sigma-\Lambda_{a_2}^\Sigma\big\Vert_\ast&\Vert u_1^{loc}\Vert_{H^{\frac{1}{2}}_{00}(\Sigma)}\Vert \overline{u_2^{loc}}\Vert_{H^{\frac{1}{2}}_{00}(\Sigma)}\nonumber\\
    \geq& C\int_{B_\rho(z_\tau)\cap\Omega}\big(a_1(x)-a_2(x)\big)|x-z_\tau|^{2-2n-2m}\,dx-C.
\end{align}
Observing that for some positive constant $s>0$, any $x\in \Omega\cap B_\rho(z_\tau)$ can be uniquely written as $x=y-s\widetilde{\nu}$, with $y\in\partial\Omega$, and
\begin{equation}
    \label{eqn: s bound}
   Cs\leq d(x,\partial\Omega)\leq s, 
\end{equation}
by Taylor's theorem, we have
\begin{equation}
    \bigg|(a_1-a_2)(x)-\underset{j=0}{\overset{h}{\sum}}\frac{\partial^j}{\partial\widetilde{\nu}^j}(a_1-a_2)(y)\frac{(-s)^j}{j!}\bigg|\leq Cs^{h+1}.
\end{equation}
For any $x\in \Omega\cap B_\rho(z_\tau)$, we write
\begin{equation}
    \label{eqn: I2.3}
    (a_1-a_2)(x)\geq s^h\bigg\Vert\frac{\partial^h}{\partial\widetilde{\nu}^h}(a_1-a_2)\bigg\Vert_{L^\infty(\overline{\Sigma}_\eta)}-\underset{j=0}{\overset{h-1}{\sum}}s^j\bigg\Vert\frac{\partial^j}{\partial\widetilde{\nu}^j}(a_1-a_2)\bigg\Vert_{L^\infty(\overline{\Sigma}_\eta)}-Cs^h|x-x^0|^\alpha.
\end{equation}
Combining together \eqref{eqn: s bound}, \eqref{eqn: I2.3} with \eqref{eqn: I2.2} leads to 
\begin{align}
    \big\Vert\Lambda_{a_1}^\Sigma-\Lambda_{a_2}^\Sigma&\big\Vert_\ast\Vert u_1^{loc}\Vert_{H^{\frac{1}{2}}_{00}(\Sigma)}\Vert \overline{u_2^{loc}}\Vert_{H^{\frac{1}{2}}_{00}(\Sigma)}\nonumber\\
    \geq & C\Bigg\{\bigg\Vert\frac{\partial^h}{\partial\widetilde{\nu}^h}(a_1-a_2)\bigg\Vert_{L^\infty(\overline{\Sigma}_\eta)}\int_{B_\rho(z_\tau)\cap\Omega}|x-z_\tau|^{2-2n-2m}d(x,\partial\Omega)^h\, dx\nonumber\\
    &-\underset{j=0}{\overset{h-1}{\sum}}\bigg\Vert\frac{\partial^j}{\partial\widetilde{\nu}^j}(a_1-a_2)\bigg\Vert_{L^\infty(\overline{\Sigma}_\eta)}\int_{B_\rho(z_\tau)\cap\Omega}|x-z_\tau|^{2-2n-2m}d(x,\partial\Omega)^j\, dx\nonumber\\
    &-\int_{B_\rho(z_\tau)\cap\Omega}|x-x^0|^\alpha|x-z_\tau|^{2-2n-2m}d(x,\partial\Omega)^h\,dx\Bigg\}-C.
\end{align}
Recalling that $C|x-z_\tau|\leq d(x,\partial\Omega)\leq|x-z_\tau|$ for any $x\in B_\rho(z_\tau)\cap\Omega$, we have
\begin{align}
    \label{eqn: I2.4}
    \bigg\Vert\frac{\partial^h}{\partial\widetilde{\nu}^h}(a_1-&a_2)\bigg\Vert_{L^\infty(\overline{\Sigma}_\eta)}\int_{B_\rho(z_\tau)\cap\Omega}|x-z_\tau|^{2-2n-2m+h}\, dx\nonumber\\
    \leq & C\Bigg\{\underset{j=0}{\overset{h-1}{\sum}}\bigg\Vert\frac{\partial^j}{\partial\widetilde{\nu}^j}(a_1-a_2)\bigg\Vert_{L^\infty(\overline{\Sigma}_\eta)}\int_{B_\rho(z_\tau)\cap\Omega}|x-z_\tau|^{2-2n-2m+j}\, dx\nonumber\\
    &+\int_{B_\rho(z_\tau)\cap\Omega}|x-x^0|^\alpha|x-z_\tau|^{2-2n-2m+h}\,dx\nonumber\\
    &+\big\Vert\Lambda_{a_1}^\Sigma-\Lambda_{a_2}^\Sigma\big\Vert_\ast\Vert u_1^{loc}\Vert_{H^{\frac{1}{2}}_{00}(\Sigma)}\Vert \overline{u_2^{loc}}\Vert_{H^{\frac{1}{2}}_{00}(\Sigma)}\Bigg\}+C.
\end{align}
We can estimate the integrals on the right hand side of \eqref{eqn: I2.4} by observing that \newline $B_\rho(z_\tau)\cap\Omega\subset\{C\tau\leq|x-z_\tau|\leq2\tau_0\}$, therefore we obtain
\begin{align}
    \label{eqn: I2.5}
   \bigg\Vert\frac{\partial^h}{\partial\widetilde{\nu}^h}(a_1-&a_2)\bigg\Vert_{L^\infty(\overline{\Sigma}_\eta)}\int_{B_\rho(z_\tau)\cap\Omega}|x-z_\tau|^{2-2n-2m+h}\, dx\nonumber\\ 
   \leq & C\Bigg\{\underset{j=0}{\overset{h-1}{\sum}}\bigg\Vert\frac{\partial^j}{\partial\widetilde{\nu}^j}(a_1-a_2)\bigg\Vert_{L^\infty(\overline{\Sigma}_\eta)}\tau^{2-n-2m+j}+ \tau^{2-n-2m+h+\alpha}\nonumber\\
   &+\big\Vert\Lambda_{a_1}^\Sigma-\Lambda_{a_2}^\Sigma\big\Vert_\ast\Vert u_1^{loc}\Vert_{H^{\frac{1}{2}}_{00}(\Sigma)}\Vert \overline{u_2^{loc}}\Vert_{H^{\frac{1}{2}}_{00}(\Sigma)}\Bigg\}+C.
\end{align}
The integral on the left hand side of \eqref{eqn: I2.5} can be estimated from below as 
\begin{equation}
    \label{eqn: lower}
    \int_{B_\rho(z_\tau)\cap\Omega}|x-z_\tau|^{2-2n-2m+h}\, dx\geq C\tau^{2-n-2m+h},
\end{equation}
(see \cite[p.~66]{MSalo}).
Recalling that by our induction argument, we have
\begin{equation}
    \label{eqn: ind re}
    \bigg\Vert\frac{\partial^j}{\partial\widetilde{\nu}^j}(a_1-a_2)\bigg\Vert_{L^\infty(\overline{\Sigma}_\eta)}\leq C\Vert\Lambda_{a_1}^\Sigma-\Lambda_{a_2}^\Sigma\Vert_\ast^{\delta_j},\quad \textrm{for any }j,\quad j\leq h-1,
\end{equation}
by combining \eqref{eqn: lower} and \eqref{eqn: ind re} and the $H^{\frac{1}{2}}_{00}(\Sigma)$-norms of $u_1^{loc}$ and $u_2^{loc}$ (see \cite{A, A-G1}), we get
\begin{align}
     \bigg\Vert\frac{\partial^h}{\partial\widetilde{\nu}^h}(a_1-a_2)\bigg\Vert_{L^\infty(\overline{\Sigma}_\eta)}\!\!\tau^{2-n-2m+h}\leq& C\bigg\{\underset{j=0}{\overset{h-1}{\sum}}\big\Vert\Lambda_{a_1}^\Sigma-\Lambda_{a_2}^\Sigma\big\Vert_\ast^{\delta_j}\tau^{2-n-2m+j}+\tau^{2-n-2m+h+\alpha}\nonumber\\
     &+\tau^{2-n-2m}\big\Vert\Lambda_{a_1}^\Sigma-\Lambda_{a_2}^\Sigma\big\Vert_\ast\bigg\}+C.
\end{align}
Hence
\begin{equation}
    \label{eqn: I2.6}
    \bigg\Vert\frac{\partial^h}{\partial\widetilde{\nu}^h}(a_1-a_2)\bigg\Vert_{L^\infty(\overline{\Sigma}_\eta)}\leq C\Big\{\big\Vert\Lambda_{a_1}^\Sigma-\Lambda_{a_2}^\Sigma\big\Vert_\ast^{\delta_{h-1}}\tau^{-h}+\tau^\alpha+\tau^{n+2m-2-h}\Big\}.
\end{equation}
By choosing $m$ sufficiently large, and optimising \eqref{eqn: I2.6} with respect to $\tau$, we obtain
\begin{equation}
    \label{eqn: pre der result}
    \Big\Vert\frac{\partial^h}{\partial\widetilde{\nu}^h}(a_1-a_2)\Big\Vert_{L^\infty(\overline{\Sigma}_\eta)}\leq C \Vert\Lambda_{a_1}^\Sigma-\Lambda_{a_2}^\Sigma\Vert_\ast^{\delta_h}.
\end{equation}
An iterative use of the interpolation inequality
\begin{equation}
    \label{eqn: inter ineq}
    \Vert Df\Vert_{L^\infty(\overline{\Sigma}_\eta)}\leq C\bigg\{\bigg\Vert\frac{\partial}{\partial\widetilde{\nu}}f\bigg\Vert_{L^\infty(\overline{\Sigma}_\eta)}+\Vert f\Vert_{L^\infty(\overline{\Sigma}_\eta)}^{\frac{\alpha}{1+\alpha}}+\Vert f\Vert_{C^{1,\alpha}(\overline{U}_\eta)}^{\frac{1}{1+\alpha}}\bigg\},
\end{equation}
which holds true for every $f\in C^{1,\alpha}(\overline{U}_\eta)$, together with \eqref{eqn: pre der result}, leads to
\begin{equation}
    \label{eqn: pre der result 2}
    \Big\Vert D^h(a_1-a_2)\Big\Vert_{L^\infty(\overline{\Sigma}_\eta)}\leq C \big\Vert\Lambda_{a_1}^\Sigma-\Lambda_{a_2}^\Sigma\big\Vert_\ast^{\delta_h}.
\end{equation}
For every multi-index $\zeta$, $|\zeta|\leq h$, we have
\begin{equation}
    D^\zeta A\big(x,a(x)\big)=\underset{\gamma+\delta\leq\zeta}{\sum}P_{\gamma\delta}\big(a(x),...,D^{|\delta|}a(x)\big)D_x^\gamma D_t^\delta A\big(x,t\big)\Big|_{t=a(x)},
\end{equation}
where $P_{\gamma\delta}$ is a polynomial, and recalling $A(x,a_i(x))\in C^{h,\alpha}(\overline{\Omega}_\eta)$ for $i=1,2$, we obtain
\begin{equation}
    \big\Vert D^h(A(x,a_1(x))-A(x,a_2(x))\big\Vert_{L^\infty(\overline{\Sigma}_\eta)}\leq C\Vert a_1(x)-a_2(x)\Vert^\alpha_{C^h(\overline{\Omega}_\eta)},
\end{equation}
which, combined with \eqref{eqn: pre der result 2} concludes the proof.
\end{proof}
\bibliographystyle{plain}
\bibliography{references}
\end{document}